\documentclass[11pt]{article}

\usepackage{amsfonts,amsmath,amssymb,amsthm}
\usepackage[margin=1in]{geometry}
\usepackage{algorithm}
\usepackage[noend]{algpseudocode}
\usepackage{float}
\usepackage{graphicx} 
\usepackage{fancyvrb}

\newcommand{\ora}{\ensuremath{\vee \mathcal A}}
\newcommand{\bE}{\ensuremath{\mathbf{E}}}
\newcommand{\pmt}{\ensuremath{\Pr_{\text{MT}}}}
\newcommand{\emt}{\ensuremath{{\mathbf E}_{\text{MT}}}}
\newtheorem{theorem}{Theorem}[section]
\newtheorem{proposition}[theorem]{Proposition}
\newtheorem{lemma}[theorem]{Lemma}
\newtheorem{corollary}[theorem]{Corollary}
\newtheorem{definition}[theorem]{Definition}
\newtheorem{conjecture}[theorem]{Conjecture}

\DeclareMathOperator{\rot}{root}
\DeclareMathOperator{\srot}{semiroot}

\newtheorem*{rep@theorem}{\rep@title}
\newcommand{\newreptheorem}[2]{%
\newenvironment{rep#1}[1]{%
 \def\rep@title{#2 \ref{##1}}%
 \begin{rep@theorem}}%
 {\end{rep@theorem}}}
\makeatother

\newreptheorem{proposition}{Proposition}

\begin{document}

\title{New bounds for the Moser-Tardos distribution}
\author{David G. Harris\thanks{Department of Computer Science, University of Maryland, 
College Park, MD 20742. 
Research supported in part by NSF Awards CNS 1010789 and CCF 1422569.
Email: \texttt{davidgharris29@gmail.com}}}

\date{}

\maketitle

\VerbatimFootnotes

\begin{abstract}

\small\baselineskip=9pt The Lov\'{a}sz Local Lemma (LLL) is a probabilistic tool which has been used to show the existence of a variety of combinatorial structures with good ``local'' properties. In many cases, one wants more information about these structures, other than that they exist. In such case, the ``LLL-distribution'' can be used to show that the resulting structures have good global properties in expectation.

While the LLL in its classical form  is a statement about probability spaces, nearly all applications in combinatorics have been turned into efficient algorithms. The simplest, variable-based setting of the LLL was covered by the seminal algorithm of Moser \& Tardos (2010). This has since been extended to other probability spaces including random permutations. One can similarly define  an ``MT-distribution'' for these algorithms, that is, the distribution of the configuration they produce. Haeupler et al. (2011) showed bounds on the MT-distribution which essentially match the LLL-distribution for the variable-assignment setting; Harris \& Srinivasan showed similar results for the permutation setting. 

In this work, we show new bounds on the MT-distribution which are significantly stronger than those known to hold for the LLL-distribution. In the variable-assignment setting, we show a tighter bound on the probability of a disjunctive event or singleton event. As a consequence, in $k$-SAT instances with bounded variable occurrence, the MT-distribution satisfies an $\epsilon$-approximate independence condition asymptotically stronger than the LLL-distribution. We use this to show a nearly tight bound on the minimum implicate size of a CNF boolean formula. Another noteworthy application is constructing independent transversals which avoid a given subset of vertices; this provides a constructive analogue to a result of Rabern (2014).

In the permutation LLL setting, we show a new type of bound which is similar to the cluster-expansion LLL criterion of Bissacot et al. (2011), but is stronger and takes advantage of the extra structure in permutations. We illustrate with improved bounds on weighted Latin transversals and partial Latin transversals.
\end{abstract}

\section{Introduction}
The Lov\'{a}sz Local Lemma (LLL) is a general principle in probability theory, first introduced in \cite{lll-orig}, that shows that it is possible to avoid a collection $\mathcal B$ of bad events in a probability space $\Omega$, as long as the bad-events are not too interdependent and are not too likely. This can be used to construct a wide variety of ``scale-free'' structures,  making the LLL one of the cornerstones of the probabilistic method in combinatorics. The simplest ``symmetric'' form of the LLL states that if each bad-event has probability at most $p$, and is dependent with at most $d$ bad-events such that $e p d \leq 1$, then there is a positive probability that all the events in $\mathcal B$ are false.

The LLL shows that a configuration avoiding $\mathcal B$ exists; we might wish to know more about such configurations. One powerful extension of the LLL, introduced by Haeupler, Saha \& Srinivasan \cite{hss}, is the \emph{LLL-distribution}; namely the distribution of $\Omega$ conditioned on avoiding $\mathcal B$. For the symmetric LLL, we  have the following bound:
\begin{proposition}[\cite{hss}]
\label{flll-dist-prop-simple}
Let $E$ be any event on the probability space $\Omega$. If the symmetric LLL criterion is satisfied, then
$$
\Pr_{\Omega}(E \mid \bigcap_{B \in \mathcal B} \overline B) \leq \Pr_{\Omega}(E) (1 + e p)^{\text{\# bad-events which affect $E$}}
$$
\end{proposition}

Thus, in a certain sense, the LLL-distribution is a mildly distorted version of the space $\Omega$. The LLL-distribution has found a number of uses in algorithms and combinatorics. Some of these techniques are developed in \cite{partial-resampling, mt-dist} which use them to construct low-weight independent transversals, partial Latin transversals and solutions to MAX-$k$-SAT, among other applications.

The LLL is stated in great generality in terms of probability spaces and events. Most applications of the LLL use a simpler formulation referred to as the \emph{variable-assignment LLL}. Here, the probability space $\Omega$ is a product space with $n$ independent variables $X(1), \dots, X(n)$.  The events in this space are boolean functions of subsets of the variables, and we say that $E \sim E'$ iff $E, E'$ overlap in some variable(s), i.e., if $\text{var}(E) \cap \text{var}(E') \neq \emptyset$.

The seminal algorithm of Moser \& Tardos \cite{moser-tardos}, (henceforth the ``MT algorithm'') turns nearly all applications of the variable-assignment LLL into corresponding polynomial-time algorithms. Furthermore one may define the \emph{MT-distribution}, namely the distribution on the variables at the termination of the MT algorithm. As shown in \cite{hss}, the same bound holds for the LLL-distribution as for the MT-distribution:
\begin{proposition}[\cite{hss}]
\label{fmt-dist-prop-simple}
Let $E$ be any event in the probability space $\Omega$. If the symmetric LLL criterion is satisfied, then the probability of $E$ in the MT-distribution is upper-bounded by
$$
\pmt(E) \leq \Pr_{\Omega}(E) (1 + e p)^{\text{\# bad-events which affect $E$}}
$$
\end{proposition}

Although the variable-assignment LLL is by far the most common setting in combinatorics, there are other probability spaces for which a generalized form of the LLL, known as the \emph{Lopsided Lov\'{a}sz Local Lemma} or \emph{LLLL}, applies. This was introduced by Erd\H{o}s \& Spencer \cite{erdos-spencer}, which showed that it applies to the probability space of the uniform distribution on permutations, and used this for a construction involving Latin transversals. We refer to this setting as the \emph{permutation LLL}; it has since been used for a variety of other problems. The LLLL also covers probability spaces such as hamiltonian cycles and matchings of the complete graph \cite{quest}. 

Algorithms similar to MT have also been developed for general probability spaces, such as \cite{harvey, achlioptas, harris-obl}; nearly all applications of the LLLL now have corresponding polynomial-time algorithms. One may define an analogous ``MT-distribution'', i.e. the distribution at the algorithm termination. Results of \cite{ilio, kolmogorov} show that for many of these settings, specifically those which have an object known as a ``commutative resampling oracle,'' the MT-distribution has similar properties to the LLL-distribution:
\begin{proposition}[\cite{ilio}]
\label{fmt-dist-prop-simple-perm}
If a probability space $\Omega$ has a commutative resampling oracle, then for any event $E$ on the space $\Omega$ we have:
$$
P_{\text{MT}}(E) \leq \Pr_{\Omega}(E) (1 + e p)^{\text{\# bad-events which affect $E$}}
$$
\end{proposition}

This principle can be used to efficiently construct an number of combinatorial objects, such as partial Latin transversals \cite{mt-dist, ilio}.

Although the MT-distribution and LLL-distribution are not the same, for the most part Propositions~\ref{flll-dist-prop-simple},  \ref{fmt-dist-prop-simple}, \ref{fmt-dist-prop-simple-perm} have been essentially the only tools available to obtain bounds on them. Since these three propositions give the same bounds, the MT-distribution and LLL-distributions might as well be the same.\footnote{There are more precise forms of all of these results, which we will encounter later in this paper. The precise formulations in all three settings give exactly the same formula.}

In this paper, we will analyze the MT-distribution more carefully. We show that the MT-distribution has smaller distortion than previously known compared to the original space $\Omega$. This is another demonstration that the MT algorithm is more than a ``constructive'' form of the LLL.

\subsection{Outline and results}
In Section~\ref{flll-sec}, we provide a review of the LLL and MT algorithm. 

In Section~\ref{fmt-disj-sec}, we analyze the MT-distribution for disjunctive events (that is, events of the form $A_1 \cup \dots \cup A_m$). We show that in the MT-distribution, the probability of the disjunction $A_1 \cup \dots \cup A_m$ is \emph{significantly less} than the sum of the probabilities of $A_1, \dots, A_m$. For instance, if each $A_i$ affects at most $r$ bad-events, then for the symmetric LLL,
$$
\pmt(A_1 \cup \dots \cup A_m ) \leq \Pr_{\Omega}(A_1 \cup \dots \cup A_m) (1 + e p)^{r}
$$

By contrast, arguments based on Proposition~\ref{fmt-dist-prop-simple} would give the weaker bound
$$
\pmt(A_1 \cup \dots \cup A_m ) \leq \min \Bigl( \Pr_{\Omega}(A_1 \cup \dots \cup A_m) (1 + e p)^{m r}, \sum_i \Pr_{\Omega}(A_i) (1 + e p)^r \Bigr)
$$

The proof is based on a connection between the MT distribution and the resampling table (a coupling construction used to analyze the MT algorithm.) The entries of the resampling table are independent, and we show that some of this independence is preserved in the MT distribution.

In Section~\ref{fmt-sing-sec}, we analyze the MT-distribution for singleton events (that is, events of the form $X(i) \in D$). Again, we show that the probability of such an event can be significantly less than a generic bound for the LLL-distribution. For instance, in the symmetric LLL, if variable $i$ appears in bad-events $B_1, \dots, B_m$, then in the MT distribution we would have
$$
\pmt(X(i) \in D) \leq \Pr_{\Omega}(X(i) \in D) (1 + e^{e p d} \Pr_{\Omega}( \bigcup_{i=1}^m B_i))
$$

By contrast, Proposition~\ref{fmt-dist-prop-simple} would give the weaker bound
$$
\pmt( X(i) \in D ) \leq \Pr_{\Omega}(X(i) \in D) (1 + e^{e p d} \sum_{i=1}^m \Pr_{\Omega}(B_i))
$$

In Section~\ref{ksat-sec}, we apply these results to $k$-SAT instances in which each variable appears in at most $L$ clauses. A classical application of the LLL and MT algorithm is to show that such instances are satisfiable when $L$ sufficiently small as a function of $k$. We show that the resulting MT-distribution is close to the probability space consisting of $n$ independent random bits. More precisely, we show that the MT-distribution is $\epsilon$-approximately, $j$-wise independent for $j \leq k$ and $\epsilon = \Theta(L 2^{-k})$. By contrast, the LLL-distribution appears to satisfy an asymptotically weaker condition with $\epsilon = \Theta(j L 2^{-k})$.

As a consequence of this, we show an extremal bound for boolean logic: for any CNF formula $\Phi$, in which every clause has size at least $k$ and every variable appears in at most $L \leq \frac{2^k}{e k}$ clauses, the minimum implicate size of $\Phi$ is $k - \lfloor \log_2 (e L) \rfloor$. This bound is asymptotically tight (up to an constant additive term). We are not aware that this property has been studied before.

In Section~\ref{find-trans-sec}, we analyze independent transversals, an important construction in graph theory. Given a graph $G$ and a subset of vertices $L \subseteq V$, we show that as long as $L$ is small enough then the MT algorithm has a good probability of producing an independent transversal disjoint to $L$. This gives a constructive version of a result of Rabern \cite{rabern}. This also significantly generalizes results of Harris \& Srinivasan \cite{partial-resampling}, which showed this only for $L$ a singleton set. 

In Section~\ref{fperm-lll-sec}, we show tighter bounds for the MT distribution corresponding to the permutation LLL. In \cite{prev-lll}, Harris \& Srinivasan developed a variant of the MT algorithm, and showed bounds on the corresponding MT-distribution it produced.  We describe a new method of analyzing their algorithm in terms of a more succinct or ``compressed'' history for why any given event $E$ became true. This analysis is inspired by Harris \cite{harris-llll}, which constructed similar (but simpler) compressed trees for the variable-assignment LLLL. This leads to better bounds on the distribution of the configuration at the algorithm termination.

In Section~\ref{latin-sec} we apply our permutation LLL bounds to constructions involving Latin transversals.  Consider an $n \times n$ matrix $A$, in which all the entries are assigned a color and each color appears at most $\Delta$ times in the matrix. In this setting, we obtained two improved bounds. First, suppose that $\Delta \leq \frac{27 n}{256}$, and we are given a weighting function $w: [n] \times [n] \rightarrow \mathbb R_{\geq 0}$. (We use here the standard notation $[n] = \{1, \dots, n \}$.) In this case, we can find a Latin transversal with weight at most $\tfrac{5}{3} \sum w(x,y)/n$.  By contrast, the LLL-distribution would only show the existence of Latin transversal of weight $\tfrac{16}{9} \sum w(x,y)/n$. Our second result is that when $\Delta > \frac{27 n}{256}$, then we construct partial Latin transversals which are larger than previous constructions in Harris \& Srinivasan \cite{mt-dist}.

\subsection{Comparison with other LLL sampling results}
Recently, a number of paper have shown that some variants of the MT algorithm can be used to sample (almost) uniformly from the LLL distribution for some problem instances \cite{moitra, guo}. In fact, in some cases the MT-distribution and the LLL-distribution are the same. Thus the MT algorithm can be used, for instance, to count the number of sink-free orientations in a graph.  These results are exciting, but we emphasize that our result is pointing in a different direction.

For many applications in combinatorics, the original, unconditioned probability space $\Omega$ is really the ``ideal'' probability space. The reason for this is that  the LLL-distribution is complex and we have little information about it, while $\Omega$ itself may be very simple. Of course, it is impossible to sample from $\Omega$ exactly and simultaneously avoid $\mathcal B$. However, our results show that the MT distribution has similar marginals to the original space $\Omega$, sometimes to a greater extent than the LLL distribution seems to. So it is almost as ``nice'' as $\Omega$, and arguably ``nicer'' than the LLL distribution itself.

\section{Background on the LLL}
\label{flll-sec}
To formally state the LLL, we define a \emph{dependency graph} for $\mathcal B$ to be a graph $G$ on vertex set $\mathcal B$, which satisfies the property that for all $B \in \mathcal B$ and all $S \subseteq \mathcal B - N(B)$ we have $\Pr_{\Omega}(B \mid \overline S) \leq \Pr_{\Omega}(B)$.\footnote{Such a graph is sometimes called a \emph{lopsidependency} graph because it is possible to have $\Pr_{\Omega}(B \mid \overline S) < \Pr_{\Omega}(B)$ strictly. The distinction will not be important for us so we use the simpler terminology.}  Here, we use the notation $N(B)$ for the \emph{inclusive} neighborhood of $B$ and we write $\overline S$ as a shorthand for the event $\bigcap_{B \in S} \overline B$.  We say that $I \subseteq \mathcal B$ is \emph{stable} if there are no distinct pairs $B, B' \in I$ with $B \in N(B')$; that is, $I$ is an independent set of $G$.

There is a natural choice for the dependency graph $G$ in the variable-assignment setting, by placing an edge on $B, B'$ if $B \sim B'$, that is, if $\text{var}(B) \cap \text{var}(B') \neq \emptyset$.

The dependency graph for the the permutation LLL is less obvious. Here, the probability space $\Omega$ is the uniform distribution on permutations $\pi$ from the symmetric group $S_n$.  A bad-event should be a monomial event, of the form $B \equiv \pi(x_1) = y_1 \wedge \dots \wedge \pi(x_k) = y_k$. For such bad-events, Erd\H{o}s \& Spencer \cite{erdos-spencer} showed that one can obtain a dependency graph $G$ by having an edge on events $B, B'$ if $\Pr_{\Omega}(B \cap B') = 0$.

Shearer \cite{shearer} stated the strongest possible criterion that can be given in terms of the dependency graph $G$ and probabilities $\Pr_{\Omega}$ for the bad-events, in order to ensure a positive probability that no bad-events occur. For any set $I \subseteq \mathcal B$, define
$$
Q_{\mathcal B} (I) = \sum_{\substack{J: I \subseteq J \subseteq \mathcal B\\\text{$J$ stable}}}(-1)^{|J| - |I|} \prod_{B \in J} \Pr_{\Omega}(B)
$$

Shearer showed that as long as $Q_{\mathcal B}(I) > 0$ for all stable sets $I$ then $\Pr(\overline {\mathcal B}) > 0$. We say in this case that \emph{Shearer's criterion is satisfied.} Furthermore, in this case $\Pr( \overline J ) \geq Q_{\mathcal B}(J)$ for all $J \subseteq \mathcal B$. Following \cite{kolipaka}, we define the \emph{measure} function $\mu: 2^{\mathcal B} \rightarrow [0,\infty)$ by
$$
\mu_{\mathcal B}(I) = \frac{ Q_{\mathcal B}(I) }{Q_{\mathcal B}( \emptyset) }
$$
for a subset $I \subseteq \mathcal B$. As the Shearer criterion is satisfied, $\mu_{\mathcal B}(I)$ is a well-defined and non-negative real number. If $I$ is not stable, then  $Q_{\mathcal B}(I) = 0$ and so $\mu_{\mathcal B}(I) = 0$.  For any $B \in \mathcal B$, we also write $\mu_{\mathcal B}(B) = \mu_{\mathcal B}( \{B \})$.

The Shearer criterion, while useful theoretically, is difficult to use in practice. Often, approximations are used. We summarize three of the most common ones here:
\begin{proposition}
\begin{enumerate}
\item (Symmetric LLL) If $\Pr_{\Omega}(B) \leq p$ and $|N(B)| \leq d$ for all $B$, where $e p d \leq 1$, then the Shearer criterion is satisfied and $\mu(B) \leq e p(B)$ for all $B$.

\item (Asymmetric LLL) If there is a weighting function $x: \mathcal B \rightarrow [0,1]$, satisfying the condition
$$
\forall B \in \mathcal B \qquad \Pr_{\Omega}(B) \leq x(B) \prod_{\substack{A \in N(B) - \{B \}}} (1 - x(A))
$$
then the Shearer criterion is satisfied, and $\mu(B) \leq \frac{x(B)}{1 - x(B)}$ for all $B$.

\item (Cluster-expansion criterion \cite{bissacot}) If there is a weighting function $\tilde \mu: \mathcal B \rightarrow [0,\infty)$ satisfying the condition
$$
\forall B \in \mathcal B \qquad \tilde \mu(B) \geq \Pr_{\Omega}(B) \sum_{\substack{I \subseteq N(B)\\\text{$I$ stable}}} \prod_{A \in I} \tilde \mu(A)
$$
then the Shearer criterion is satisfied, and $\mu(B) \leq \tilde \mu(B)$ for all $B$.
\end{enumerate}
\end{proposition}

The definition of dependency can be extended to any event $E$ (not necessarily in $\mathcal B$). We let $N(E) \subseteq \mathcal B$ be a set with the property that $\Pr_{\Omega}(E \mid \overline S) \leq \Pr_{\Omega}(E)$ for all sets $S \subseteq \mathcal B - N(E)$.  In the variable-assignment or permutation LLL settings, for instance, we define $N(E)$ to be the set of bad-events $B$ with $B \sim E$.

There is a connection between the Shearer criterion and LLL-distribution, which can be expressed by defining the following quantity:
$$
\Psi_{\mathcal B}(E) =  \sum_{J \subseteq N(E)}  \mu_{\mathcal B}(J) 
$$

We define $\mathcal B[E]$ to be the set of events $B \in \mathcal B$ such that $\Pr_{\Omega}(E \mid B) < 1$. In Appendix~\ref{mt-dist-result}, we show the following fundamental bound on the LLL distribution:
\begin{theorem}
\label{lll-dist-theta-thm}
For any event $E$, we have $\Pr_{\Omega} (E \mid \overline{\mathcal B}) \leq \Pr_{\Omega}(E) \Psi_{\mathcal B[E]}(E)$.
\end{theorem}

Kolipaka \& Szegedy \cite{kolipaka} introduced a key analytical tool which ties together the Shearer criterion, LLL-distribution, and MT-distribution. Given stable sets $S_1, \dots, S_{\ell} \subseteq \mathcal B$, we say that $S = \langle S_1, \dots, S_{\ell} \rangle$ is a \emph{stable-set sequence} if $S_i \subseteq \bigcup_{B \in S_{i-1}} N(B)$ for all $i = 2, \dots, \ell$.  We define the \emph{weight} of $S$, denoted $w(S)$, by $w(S) = \prod_{i=1}^{\ell} \prod_{B \in S_i} \Pr_{\Omega}(B)$.  For any stable set $I \subseteq \mathcal B$, we define $\text{Stab}(I)$ to be the set of stable-set sequences of the form $\langle I, S_2, \dots, S_{\ell} \rangle$ (that is, the first set in the sequence is $I$), and we define $\text{Stab}(I) = \emptyset$ if $I$ is not stable. 
\begin{proposition}[\cite{kolipaka}]
\label{ftotal-weight-prop}
For any set $J \subseteq \mathcal B$, we have $\sum_{S \in \text{Stab}(J)} w(S) = \mu_{\mathcal B}(J)$.
\end{proposition}

Haeupler \& Harris \cite{hh} discussed a closely-related object called a \emph{witness DAG} (abbreviated wdag).  A wdag $G$ is defined to be a DAG, whose nodes are labeled by events in the probability space $\Omega$, such that for any two nodes $v_1, v_2$ labeled $E_1, E_2$, we have an edge (in either direction) on $v_1, v_2$ iff $E_1 \sim E_2$. We use $L(v)$ to denote the label of node $v$. We define the \emph{weight} of $G$ by $w(G) = \prod_{v \in G} \Pr_{\Omega}( L(v) )$.

We say that a wdag $G$ is \emph{proper} if all the node labels of $G$ come from $\mathcal B$. There is a one-to-one correspondence between proper wdags and stable-set sequences as discussed in \cite{hh}; for example, we have the following result:
\begin{proposition}[\cite{hh}]
\label{ac1}
For any set $I \subseteq \mathcal B$, the total weight of all proper wdags with sink nodes $v_1, \dots, v_s$ such that $I = \{L(v_1), \dots, L(v_s) \}$, is at most $\mu_{\mathcal B}(I)$.
\end{proposition}

Given a wdag $G$ and an event $E$, we say that a node $v \in G$ is \emph{relevant} to $E$ if there is a path in $G$ from $v$ to some node $w$ with $L(w) \sim E$. 

We let $\Gamma(A)$ denote the set of wdags $\tau$ with a single sink node labeled $A$; we call this sink node the \emph{root} of $\tau$.  For a wdag $G$ and an event $A$, we may define a new wdag $G * A \in \Gamma(A)$ via the following construction. We begin with the induced subgraph $G[X]$, where $X$ are the set of vertices in $G$ relevant to $A$, and we then add an additional node $r$ labeled $A$. Finally, we add an edge from each node $v$ to $r$ if $L(v) \sim A$.

\subsection{The Moser-Tardos algorithm}
\label{fmt-sec}

Moser \& Tardos in \cite{moser-tardos} described their MT algorithm to construct a configuration $X$ avoiding $\mathcal B$ for the variable-assignment LLL setting:

\begin{algorithm}[H]
\centering
\caption{The MT Algorithm}
\begin{algorithmic}[1]
\State  Draw $X \sim \Omega$. (We refer to this $X$ as the \emph{initial configuration})
\While{there is some true bad-event}
\State {Choose some true bad-event $B \in \mathcal B$ arbitrarily}
\State Resample all the variables involved in $B$ from the distribution $\Omega$.
\EndWhile
\end{algorithmic}
\end{algorithm}

\begin{proposition}[\cite{kolipaka}]
\label{fmt-main-prop}
Suppose that Shearer's criterion is satisfied. Then the MT algorithm terminates with probability one; the expected number of resamplings of any $B$ is at most $\mu(B)$.
\end{proposition}

We define the \emph{MT-distribution} to be the distribution on the variables at the termination of the MT algorithm. We use $\pmt$ to refer to probabilities of events in this space.\footnote{Technically, the MT-distribution is not well-defined unless one specifies some fixed, possibly randomized, rule for selecting which bad-event to resample. In this paper, we assume that some such rule has been chosen arbitrarily; all of our results hold for any such choice of resampling rule.} Using ideas from \cite{kolipaka} and \cite{hss} one can show the following bound, which exactly matches Theorem~\ref{lll-dist-theta-thm}:
\begin{proposition}
\label{fmt-dist-prop1}
For any event $E$ in the probability space $\Omega$, we have $\pmt(E) \leq \Pr_{\Omega}(E) \Psi_{\mathcal B[E]}(E)$
\end{proposition}

The resampling table is another key analytical tool of \cite{moser-tardos} for analyzing the MT algorithm. It is constructed as follows. For each $i \in [n]$, we draw an infinite string of values $R(i,1), R(i,2), \dots $, each entry $R(i,j)$ being drawn independently according to the distribution $\Omega$. We then use this as our source of randomness during the MT algorithm: the initial configuration sets $X(i) = R(i, 1)$ for all $i$. Whenever we need to resample a variable $i$, we take the next entry from $R(i)$; for instance, on the first resampling of $X_i$, we would set $X_i = R(i,2)$. So the MT algorithm becomes essentially deterministic after we have drawn $R$.

The main connection between wdags, the resampling table, and the MT algorithms, comes from the following construction. Suppose that we run the MT algorithm (possibly not to completion), resampling bad events $B_1, \dots, B_T$. We define the \emph{full wdag} $\hat G_T$ to be a graph with nodes $v_1, \dots, v_T$ labeled $B_1, \dots, B_T$, and an edge from $v_i$ to $v_j$ if $i < j$ and $B_i \sim B_j$.

If some event $A$ is true at time $T$, we also define $\hat \tau^{T,A} \in \Gamma(A)$ as $\hat G_T * A$; we say that a wdag $\tau$ \emph{appears} if $\hat \tau^{T,A} = \tau$ for any value of $T$.

We next derive a necessary condition for a given wdag $\tau$ to appear, as a function of the resampling table $R$. For any node $v \in \tau$ and variable $i \in [n]$, define $\rho(\tau, v, i)$ to be the number of nodes $w \in \tau$ such that $w$ has a path to $v$ and $i \in \text{var}(L(w))$ (this includes $w = v$).  We also define the configuration $X_{\tau,v}^R$ by setting 
$$
X^R_{\tau,v}(i) = R(i, \rho(\tau, v, i)) \qquad \text{for all $i \in [n]$}
$$ 
We say that $\tau$ is \emph{compatible} with $R$, if for all nodes $v \in \tau$, the event $L(v)$ is true on the configuration $X^R_{\tau,v}$. For any set of wdags $T$, and a given resampling table $R$, we also define $T/R$ to be the set of wdags in $T$ compatible with $R$.

The following are two key results for these objects:
\begin{proposition}[\cite{moser-tardos,hh}]
\label{mtprop1}
\begin{enumerate}
\item If the MT algorithm runs with a given resampling table $R$, then every wdag $\hat \tau^{T,A}$ is compatible with $R$.
\item For any wdag $\tau$, the probability that resampling table $R$ is compatible with $\tau$ is exactly $w(\tau)$.
\end{enumerate}
\end{proposition}

These results give a simple proof of Proposition~\ref{fmt-main-prop}. For, suppose we run the MT algorithm with a random resampling table $R$ and resample events $B_1, \dots, B_t$. For each $i = 1, \dots, t$, consider the wdag $\hat \tau^{i, B_i}$; these are all distinct and we have $\hat \tau^{i, B_i} \in \Gamma(B_i)$. So the number of resamplings of $B$ is at most number of wdags in $\Gamma(B)$ compatible with $R$, and hence the expected number of resamplings of $B$ is at most $\sum_{\tau \in \Gamma(B)} w(\tau) = \mu_{\mathcal B}(B)$.

These results also give a simple proof of Proposition~\ref{fmt-dist-prop1}. For, suppose we run the MT algorithm with a random resampling table $R$ until some event $E$ occurs. (This may already be at the initial configuration). As soon as $E$ occurs, form the corresponding $\tau = \hat \tau^{T,E}$. Evidently, $\tau$ appears; also, all its nodes are labeled by $\mathcal B[E]$, because we stop this process the first time that $E$ becomes true. Hence, we have $\hat \tau^{T,E} \in \Gamma_{\mathcal B[E]}(E)$; the total weight of all such wdags is at most $\Pr_{\Omega}(E) \Psi_{\mathcal B[E]}(E)$.

\section{Disjunctive events in the variable-assignment setting}
\label{fmt-disj-sec}
Given a set of events $\mathcal A = \{A_1, \dots, A_m \}$, in the probability space $\Omega$, we define the disjunction event $\ora = A_1 \cup A_2 \cup \dots \cup A_m$. In this section we will show an upper bound on $\pmt(\ora)$, beyond the obvious union bound $\pmt(\ora) \leq \sum_{A \in \mathcal A} \pmt(A) \leq \sum_{A \in \mathcal A} \Pr_{\Omega}(A) \Psi_{\mathcal B[A]}(A)$.

For each $A \in \mathcal A$, define $T_A$ to be the set of wdags in $\Gamma(A)$ whose non-sink nodes have labels from $\mathcal B[ \ora]$. For $\tau \in T_A$, we define $\rot(\tau)$ to be the root node of $\tau$. 

\begin{proposition}
\label{fxx1}
Suppose that we run the MT algorithm with a given resampling table $R$ until the first time $t$ that $\ora$ becomes true. (If $\ora$ is true in the initial configuration, then $t = 0$). If $A \in \mathcal A$ is true at time $t$, then $\hat \tau^{t, A} \in T_A/R$.
\end{proposition}
\begin{proof}
First, $\hat \tau^{t, A}$ is compatible with $R$ by Proposition~\ref{mtprop1}.  Suppose that $\hat \tau^{t, A} \notin T_A$,  so that $\hat \tau^{t, A}$ contains a node labeled $B$ with $\Pr(\ora \mid B) = 1$. So $B$ was resampled at some time $j < t$. This implies that $\ora$ was also true at time $j$, contradicting minimality of $t$.
\end{proof}

\begin{proposition}
  \label{fxx2}
  For $A \in \mathcal A$ we have the bound $\sum_{\tau \in T_A} w( \tau - \rot(\tau) ) \leq \Psi_{\mathcal B[ \ora]} (A)$.
\end{proposition}
\begin{proof}
  To enumerate $T_A$, observe that $\tau - \rot(\tau)$ is a proper wdag, whose sink nodes have distinct labels $I$ such that $I \subseteq N(A)$. Furthermore, the labels of all the nodes in $\tau - \rot(\tau)$ must come from $\mathcal B[ \ora]$. The result then follows from Proposition~\ref{ac1}. 
\end{proof}

The key to our improved bound will be a construction relating $T_A$ to $T_{A'}$ for $A \neq A'$. Given $\tau \in T_A$ and $A' \in \mathcal A$,  we define a new wdag $\Phi_{A'}(\tau)$ by  changing the label of its root from $A$ to $A'$, and removing any remaining nodes which no longer have a path to the new root. Formally, we define:
$$
\Phi_{A'}(\tau) = (\tau - \rot(\tau)) * A'.
$$

\begin{proposition}
\label{ftrans-prop}
Let $R$ be a fixed resampling table and let $\tau \in T_A/R$ with root $r = \rot(\tau)$. If $A'$ is true on $X^R_{\tau, r}$, then $\Phi_{A'}(\tau) \in T_{A'}/R$.
\end{proposition}
\begin{proof}
Let $\tau' = \Phi_{A'}(\tau)$. Consider any node $v$ of $\tau - r$ with $L(v) = B$. If $v$ is present in $\tau'$, then every $w \in \tau$ with a $\tau$-path to $v$ will also appear in $\tau'$ with $\tau'$-path to $v$. Thus, for all $x \in \text{var}(B)$, we have $\rho(\tau, v, x) = \rho(\tau', v, x)$ and so $X_{\tau,v}^R(x) = X_{\tau', v}^R(x)$. As $B$ holds on $X_{\tau,v}^R$ it also holds on $X_{\tau', v}^R$.

Next, consider the root $r'$ of $\tau'$ labeled $A'$, and let $x \in \text{var}(A')$. Any non-root node $v$ of $\tau$ with $x \in L(v)$ will also be present in $\tau'$. Hence $\rho(\tau, r,x) = \rho(\tau', r', x)$. By hypothesis, $A'$ holds on $X_{\tau, r}^R$ so it must also hold on $X_{\tau', r'}^R$.
\end{proof}

\begin{proposition}
\label{fx4}
Suppose that we enumerate $\mathcal A$ in an arbitrary order as $\mathcal A = \{A_1, \dots, A_m \}$. Then 
$$
\pmt(\ora) \leq \sum_{j=1}^m \Pr_{\Omega}(A_j \wedge \neg A_1 \dots \wedge \neg A_{j-1}) \Psi_{\mathcal B[ \ora]} (A_j)
$$
\end{proposition}
\begin{proof}
 Suppose that we run the MT algorithm with a randomly-chosen resampling table $R$. If $\ora$ is true at some time $t$, then by Proposition~\ref{fxx1} one of the sets $T_{A_j}/R$ is non-empty. Let $j$ be \emph{minimal} such that $T_{A_j}/R \neq \emptyset$ and let $\tau \in T_{A_j}/R$ with root $r$. It must be that $A_1, \dots, A_{j - 1}$ are all false on $X_{\tau, r}^R$; for if $A_k$ is true on $X_{\tau,r}^R$ and $k < j$, then by Proposition~\ref{ftrans-prop} we would have $\Phi_{A_k}(\tau) \in T_{A_k}/R$, contradicting minimality of $j$.

So, we see that a necessary condition for $\ora$ to occur is that there is some $\tau \in T_{A_j}/R$ with root $r$ such that $A_1, \dots, A_{j - 1}$ are all false on $X_{\tau,r}^R$. Let us fix some  $\tau \in T_{A_j}$ with root $r$, and compute the probability that $\tau$ is compatible with $R$ \emph{and} all the events $A_1, \dots, A_{j-1}$ are false on $X_{\tau, r}^R$. As $X_{\tau, r}^R \sim \Omega$, this probability is $\Pr_{\Omega}(A_j \wedge \neg A_1 \dots \wedge \neg A_{j-1})$. Also, for each non-sink node $v \in \tau$, it must be that  $L(v)$ is true on $X_{\tau, v}^R$, which has probability exactly $\Pr_{\Omega}(L(v))$. 

Finally, observe that the conditions corresponding to disjoint nodes involve distinct entries of $R$. Hence, they are independent, and the total probability that $\tau$ satisfies all these conditions is exactly equal to their product. So the overall probability is 
$$
\Pr_{\Omega}(A_j \wedge \neg A_1 \dots \wedge \neg A_{j-1}) \times \prod_{v \in \tau - r} \Pr_{\Omega}(L(v))   = \Pr_{\Omega}(A_j \wedge \neg A_1 \dots \wedge \neg A_{j-1})  w(\tau - \rot(\tau))
$$

Putting these observations together, we see that
{\allowdisplaybreaks
\begin{align*}
\pmt(\ora) &\leq \sum_{j=1}^m \sum_{\tau \in T_{A_j}} \Pr(\text{$\tau$ compatible with $R$ and $A_1, \dots, A_{j-1}$ false on $X_{\tau,r}^R$}) \\
&\leq \sum_{j=1}^m \sum_{\tau \in T_{A_j}} w(\tau) \Pr_{\Omega}(A_j \wedge \neg A_1 \wedge \dots \wedge \neg A_{j-1}) \\
&\leq \sum_{j=1}^m \Psi_{\mathcal B[ \ora]}(A_j) \Pr_{\Omega}(A_j \wedge \neg A_1 \wedge \dots \wedge \neg A_{j-1}) \qquad \text{Proposition~\ref{fxx2}} \qedhere
\end{align*}
}
\end{proof}

\begin{corollary}
\label{fcor-a}
We have $\pmt(\ora) \leq \Pr_{\Omega}(\ora) (\max_{A \in \mathcal A} \Psi_{\mathcal B[\ora]} (A))$
\end{corollary}
\begin{proof}
Arbitrarily enumerate $\mathcal A = \{A_1, \dots, A_m \}$. By Proposition~\ref{fx4}:
\begin{align*}
\Pr(\ora) &\leq \sum_{j=1}^m \Psi_{\mathcal B[\ora]} (A_j) \Pr_{\Omega}(\neg A_1 \wedge \dots \wedge \neg A_{j-1} \wedge A_j) \\
& \leq (\max_i \Psi_{\mathcal B[A]} (A_j)) \sum_{j=1}^m \Pr_{\Omega}(\neg A_1 \wedge \dots \wedge \neg A_{j-1} \wedge A_j) =  \Pr_{\Omega}(\ora) (\max_{A \in \mathcal A} \Psi_{\mathcal B[\ora]}) \qedhere
\end{align*}
\end{proof}

Corollary~\ref{fcor-a} can be simplified in the symmetric LLL setting.
\begin{corollary}
Suppose that $\Pr_{\Omega}(B) \leq p$ and $|N(B)| \leq d$ for all $B \in \mathcal B$, and that $e p d \leq 1$. Suppose that $|N(A)| \leq r$ for each $A \in \mathcal A$. Then $\pmt(\ora) \leq \Pr_{\Omega}(\ora) (1 + ep)^r$.
\end{corollary}
\begin{proof}
In the symmetric LLL setting, $\mu_{\mathcal B}(B) \leq e p$ for every $B \in \mathcal B$. Hence, for any $A$, we have $\Psi_{\mathcal B}(A) \leq \prod_{B \sim A} (1 + \mu_{\mathcal B}(B)) \leq (1 + e p)^r$.
\end{proof}

\section{Singleton events in the variable-assignment setting}
\label{fmt-sing-sec}
A \emph{singleton event} is an event of the form $A \equiv X(i) \in D$,  where $D$ is a subset of possible values for variable $i$. Often the existence of configurations with good global weight properties depends on bounding certain singleton events; for example, Harris \& Srinivasan \cite{partial-resampling} bounded the probability of singleton events in the LLL-distribution to show the existence of independent transversals with low average weight (see also Section~\ref{find-trans-sec}). 

Let us fix some $A \equiv X(i) \in D$, and define $\mathcal B'$ to the set of bad-events $B \in \mathcal B[A]$ with $i \in \text{var}(B)$. Our goal is to show an upper bound on $\pmt(A)$. We can immediately observe that $\mathcal B'$ is a clique of the dependency-graph, and so 
$\Psi_{\mathcal B[A]}(A) = 1 + \sum_{B \in \mathcal B'} \mu_{\mathcal B[A]}(B)$. Consequently, Proposition~\ref{fmt-dist-prop1} immediately gives the simple bound
\begin{equation}
\label{sb1}
\Pr_{\Omega}(A \mid \overline{\mathcal B}) \leq \Pr_{\Omega}(A) ( 1 + \sum_{B \in \mathcal B'} \mu_{\mathcal B[A]}(B) ),
\end{equation}
and the same bound holds for the MT distribution.  We will obtain a stronger bound for the MT-distribution in this section, using a similar proof strategy to the one for disjunctive events in  Section~\ref{fmt-disj-sec}.

For each $B \in \mathcal B$, define $T_B$ to be the set of wdags $\tau \in \Gamma(A)$ with the following structure: $\tau$ has a single sink node $r$ labeled $A$, and a node $s$ labeled $B$ with an edge to $r$; all other nodes $w \in \tau$ which involve variable $i$ have an edge to $s$. We refer to $r$ as the \emph{root} of $\tau$ and $s$ as the \emph{semi-root} of $\tau$. Equivalently, $T_B$ is the set of wdags of the form $(G * B) * A$.

\begin{proposition}
\label{gxx1}
Suppose that we run the MT algorithm with some resampling table $R$ until the first time $t \geq 0$ that $A$ becomes true. If $t > 0$ then $\hat \tau^{t, A} \in T_B/R$ for some $B \in \mathcal B'$.
\end{proposition}
\begin{proof}
First, $\hat \tau^{t, A}$ is compatible with $R$ by Proposition~\ref{mtprop1}. As $A$ is not true initially, it must become true by resampling some $B \in \mathcal B'$. Since this is the first time $A$ is true, it must be that $B \in \mathcal B[A]$. Hence $\hat \tau^{t,A} \in T_B$.
\end{proof}

\begin{proposition}
  \label{gxx2}
    For $B \in \mathcal B$ we have the bound $\sum_{\tau \in T_B} w( \tau - \rot(\tau) - \srot(\tau) ) \leq \Psi_{\mathcal B[A]} (B)$.
\end{proposition}
\begin{proof}
 To enumerate $T_B$, observe that any $\tau \in T_B$ has root node $r$ labeled $A$ and semi-root node $s$ labeled $B$. The graph $\tau - r - s$ is a wdag whose sink nodes are labeled by some stable set $I \subseteq N_{\mathcal B[A]} (B)$. The result then follows from Proposition~\ref{ac1}.
\end{proof}

As in Section~\ref{fmt-disj-sec}, given $\tau \in T_B$ with root $r$ and semi-root $s$ and $B' \in \mathcal B$, we define $\Phi_{B'}(\tau)$ by changing the label of the semi-root from $B$ to $B'$,  and removing any remaining nodes which no longer have a path to the semi-root. Formally:
$$
\Phi_{B'}(\tau) = ( (\tau - \rot(\tau) - \srot(\tau)) * B' ) * A
$$

\begin{proposition}
\label{gtrans-prop}
Let $R$ be a given resampling table and let $\tau \in T_B/R$. If $B'$ is true on $X_{\tau, \srot(\tau)}^R$, then $\Phi_{B'}(\tau) \in T_{B'}/R$.
\end{proposition}
\begin{proof}
Let us write $\tau' = \Phi_{B'}(\tau), r = \rot(\tau), s = \srot(\tau), r' = \rot(\tau'), s' = \srot(\tau')$.

For a node $v \in \tau' - r' - s'$, we can see that $\rho(\tau, v, x) = \rho(\tau', v, x)$ for any $x \in \text{var}(L(v))$; thus $L(v)$ holds on $X_{\tau', v}^R$ since it holds on $X_{\tau, v}^R$.

For the root, observe that $B$ and $B'$ both involve variable $i$, and so again $\rho(\tau, r, i) = \rho(\tau', r, i')$. This is the only variable affecting event $A$, since it is a singleton; so again $A$ holds on $X_{\tau', r'}^R$ since it holds on $X_{\tau, r}^R$.

Finally, for the semi-root, we see that $\rho(\tau, s, x) = \rho(\tau', s', x)$ for every $x \in \text{var}(B')$. By hypothesis, $B'$ holds on $X_{\tau, s}^R$ so it holds on $X_{\tau', s'}^R$.
\end{proof}

The following is our main result for the MT distribution here. Note that it is always at least as strong as the LLL-distribution bound (\ref{sb1}). 
\begin{theorem}
\label{gx4}
Suppose that we enumerate $\mathcal B'$ in an arbitrary order as $\mathcal B' = \{B_1, \dots, B_m \}$. Then 
$$
\pmt(A) \leq \Pr_{\Omega}(A) \Bigl( 1 + \sum_{j=1}^m \Pr_{\Omega}(B_j \wedge \neg B_1 \dots \wedge \neg B_{j-1}) \Psi_{\mathcal B[A]} (B_j) \Bigr)
$$
\end{theorem}
\begin{proof}
  Consider running the MT algorithm with a randomly chosen resampling table $R$. If $A$ becomes true at a some time $t > 0$, then by Proposition~\ref{gxx1} one of the sets $T_{B_j}/R$ is non-empty. Let $j$ be \emph{minimal} such that $T_{B_j}/R \neq \emptyset$ and let $\tau \in T_{B_j}/R$ with semi-root $s$. It must be that $B_1, \dots, B_{j - 1}$ are all false on $X_{\tau, s}^R$; for if $B_k$ is true on $X_{\tau, s}^R$ and $k < j$, then by Proposition~\ref{gtrans-prop} we would have $\Phi_{B_k}(\tau) \in T_{B_k}/R$, contradicting minimality of $j$.

So, a necessary condition for $A$ to occur is that either $A$ holds on the initial configuration, or there is some $\tau \in T_{B_j}/R$ with semi-root $s$ such that $B_1, \dots, B_{j - 1}$ are all false on $X_{\tau,s}^R$. Let us fix some  $\tau \in T_{B_j}$ with root $r$ and semi-root $s$, and compute the probability that $\tau$ is compatible with $R$ \emph{and} all the events $B_1, \dots, B_{j-1}$ are false on $X_{\tau, s}^R$. As $X_{\tau, s}^R \sim \Omega$, this probability is $\Pr_{\Omega}(B_j \wedge \neg B_1 \dots \wedge \neg B_{j-1})$. Also, for each non-sink node $v \in \tau$, it must be that  $L(v)$ is true on $X_{\tau, v}^R$. Finally, the event $A$ must hold on $X_{\tau, r}^R$.

Since the events $B_1, \dots, B_m$ all involve variable $i$, the conditions corresponding to disjoint nodes involve distinct entries of $R$. Hence, they are independent, and the total probability that $\tau$ satisfies all these conditions is exactly equal to their product, i.e.
$$
\Pr_{\Omega}(B_j \wedge \neg B_1 \dots \wedge \neg B_{j-1}) \times \Pr_{\Omega}(A) \times \prod_{v \in \tau - r - s} \Pr_{\Omega} (L(v))   = \Pr_{\Omega}(A) \Pr_{\Omega}(\neg B_1 \wedge \dots \wedge \neg B_{j-1} \wedge B_j) w(\tau - r - s) 
$$

The probability that $A$ holds on the initial configuration is just $\Pr_{\Omega}(A)$, so in total we have 
{\allowdisplaybreaks
\begin{align*}
\pmt(A) &\leq \Pr_{\Omega}(A) + \sum_{j=1}^m \sum_{\tau \in T_{B_j}} \Pr(\text{$\tau$ compatible with $R$ and $B_1, \dots, B_{j-1}$ false on $X_{\tau, s}^R$}) \\
&\leq \Pr_{\Omega}(A) + \sum_{j=1}^m \sum_{\tau \in T_{B_j}} \Pr_{\Omega}(A) \Pr_{\Omega}(\neg B_1 \wedge \dots \wedge \neg B_{j-1} \wedge B_j) w(\tau - \rot(\tau) - \srot(\tau)) \\
&\leq \Pr_{\Omega}(A) \Bigl(1 + \sum_{j=1}^m \Pr_{\Omega}(\neg B_1 \wedge \dots \wedge \neg B_{j-1} \wedge B_j)  \Psi_{\mathcal B[A]}(B_j) \Bigr) \qquad \text{Proposition~\ref{gxx2}} \qedhere
\end{align*}
}
\end{proof}

We have two simpler corollaries.
\begin{corollary}
\label{gcor-a}
We have $\pmt(A) \leq \Pr_{\Omega}(A) \Bigl( 1 +  \Pr_{\Omega}(\vee \mathcal B')  \bigl(\max_{B \in \mathcal B'} \Psi_{\mathcal B[A]} (B)\bigr)\Bigr)$
\end{corollary}
\begin{proof}
Enumerate $\mathcal B' = \{B_1, \dots, B_m \}$ and use Theorem~\ref{gx4}:
\begin{align*}
\pmt(A) &\leq \Pr_{\Omega}(A) \bigl( 1 + \sum_{j=1}^m \Pr_{\Omega}(B_j \wedge \neg B_1 \dots \wedge \neg B_{j-1}) \Psi_{\mathcal B[A]} (B_j) \bigr) \\
& \leq \Pr_{\Omega}(A) \Bigl( 1 + (\max_j \Psi_{\mathcal B[A]} (B_j)) \sum_{j=1}^m \Pr_{\Omega}(\neg B_1 \wedge \dots \wedge \neg B_{j-1} \wedge B_i) \Bigr) \\
&= \Pr_{\Omega}(A) \bigl( 1 + (\max_{B \in \mathcal B'} \Psi_{\mathcal B[A]} (B)) \Pr_{\Omega}( \vee \mathcal B' ) \qedhere
\end{align*}
\end{proof}

\begin{corollary}
Suppose that $\Pr_{\Omega}(B) \leq p$ and $|N(B)| \leq d$ for all $B \in \mathcal B$, and that $e p d \leq 1$. Then $\pmt(A) \leq \Pr_{\Omega}(A) (1 + e^{e p d} \Pr_{\Omega}(\vee \mathcal B'))$.
\end{corollary}
\begin{proof}
In the symmetric LLL setting, $\mu_{\mathcal B}(B) \leq e p$ for every $B \in \mathcal B$. Hence, for $B \in \mathcal B'$, we have $\Psi_{\mathcal B}(B) \leq \prod_{B' \sim B} (1 + \mu_{\mathcal B}(B')) \leq (1 + e p)^d \leq e^{e p d}$.
\end{proof}

\section{The MT-distribution for $k$-SAT assignments}
\label{ksat-sec}
The probability space $\Omega^*$ consisting of $n$ independent fair coins can be considered the ``ideal'' probability space on $\{0, 1\}^n$. One powerful tool used to measure the divergence of a probability space $\Omega$ on $\{0, 1\}^n$ from $\Omega^*$ is the notion of \emph{$\epsilon$-approximate $j$-independence}, introduced by \cite{naor}. We say that $\Omega$ is $\epsilon$-approximately $j$-independent if, for any indices $1 \leq i_1 < \dots < i_j \leq n$ and any bits $y_1, \dots, y_j \in \{0, 1 \}^n$, we have
\begin{equation}
\label{approx-formula}
\bigl| \Pr_{\Omega} \bigl ( X(i_1) = y_1 \wedge \dots \wedge X(i_j) = y_j \bigr) - 2^{-j} \bigr| \leq \epsilon
\end{equation}

Many randomized algorithms and constructions can replace a supply of independent bits with a probability space possessing $\epsilon$-approximate $j$-independence, for appropriate values of $\epsilon$ and $j$. 

We next show that (under appropriate conditions) the MT-distribution approximates $\Omega^*$, in the sense that the MT-distribution is $j$-wise $\epsilon$-approximate independent.  Thus, for algorithmic applications which require limited independence, the MT-distribution could be used as the randomness source. Furthermore, the MT-distribution appears to be asymptotically closer to $\Omega^*$ than  the LLL-distribution (at least as far as we can tell using the generic bounds known for the LLL-distribution).

Concretely, we consider the application of the LLL to a $k$-SAT problem. Here, we are given a conjunction of $m$ clauses in $n$ binary variables. Each clause is a disjunction of $k$ distinct literals, i.e. it has the form $X(i_1)= y_1 \vee \dots \vee X(i_k) = y_k$. We wish to find an assignment of the values $X(1), \dots, X(n)$ which simultaneously satisfies all clauses. A classical application of the LLL, and the MT algorithm, is to $k$-SAT instances in which each variable appears at most $L$ times. In particular, Harris \cite{harris-llll} shows that the bound $L \leq \frac{2^{k+1} (1 - 1/k)^k}{k-1} - \frac{2}{k}$ suffices and Gebauer, Szab\'{o}, Tardos \cite{gst} shows that this is asymptotically tight (up to second-order terms).

For smaller values of $L$, we can find a satisfying assignment by applying the LLL, where the space $\Omega$ is defined by setting $\Pr(X(i) = 0) = \Pr(X(i) = 1) = 1/2$. For each clause $C$, we have a bad-event that $C$ is violated. This bad-event has probability $p = 2^{-k}$, and depends upon at most $L k$ others. Thus, as long as $L \leq \frac{2^k}{e k}$, the symmetric LLL criterion is satisfied, and the MT algorithm finds a satisfying assignment. Furthermore, for each bad-event $B$, we have $\mu_{\mathcal B}(B) \leq e p = e 2^{-k }$.

\begin{theorem}
\label{dist-thm}
Let $1 \leq j \leq k$ and $L \leq \frac{2^k}{e k}$. The MT-distribution is $\epsilon$-approximate, $j$-wise independent for $\epsilon = e L 2^{-k}$.
\end{theorem}
\begin{proof} 
Observe that for any singleton event $A \equiv X(i) = y$, a stable set of neighbors of $A$ is either the empty set, or is a singleton set consisting of a clause involving variable $i$. Thus, we have $\Psi_{\mathcal B}(A) \leq 1 + L e p$, where we recall that $p = 2^{-k}$. Similarly, for any monomial event $A$ in $r$ variables, we have $\Psi_{\mathcal B}(A) \leq (1 + L e p)^r$.

Now consider some $j$-tuple $i_1, \dots, i_j$ and some arbitrary $y_1, \dots, y_j \in \{0, 1\}^n$. Let $E$ be the event $X(i_1) = y_1 \wedge \dots \wedge X(i_j) = y_j$. We want to show $2^{-j} - \epsilon \leq \Pr(E) \leq 2^{-j} + \epsilon$. 

For the upper bound:
\begin{align*}
\pmt( E) &\leq 2^{-j} \Pr_{\Omega}(E) \Psi_{\mathcal B}( E ) \leq 2^{-j} (1 + L e p)^j \leq 2^{-j} e^{L j e 2^{-k}} \leq 2^{-j} + 2^{-j} e (e-1) j L 2^{-k}
 \end{align*}
 
For the lower bound:
\begin{align*}
\pmt(E) &= 1 - \pmt( X(i_1) = 1 - y_1 \vee \dots \vee X(i_j) = 1 -y_j) \\
&\geq 1 - \Pr_{\Omega}( X(i_1) = 1 - y_1 \vee \dots \vee X(i_j) = 1 -y_j) (1 + L e p) \qquad \text{Corollary~\ref{fcor-a}} \\
&= 1 - (1 - 2^{-j}) (1 + L e p) \geq 2^{-j} - e L 2^{-k}
\end{align*}

In total, the MT-distribution is $\epsilon$-approximately $j$-wise independent for 
\[
\epsilon = \max( 2^{-j} j e (e-1) L 2^{-k}, e L 2^{-k}) = e L 2^{-k}. \qedhere
\]
\end{proof}

We do not know to what extent these results hold for the LLL-distribution. If we directly use Proposition~\ref{flll-dist-prop-simple}, then the best  lower-bound we can obtain is 
\begin{align*}
\Pr(E \mid \overline{\mathcal B}) &= 1 - \Pr( X(i_1) = 1 - y_1 \vee \dots \vee  X(i_j) = 1 -y_j \mid \overline{\mathcal B}) \\
&\geq 1 - \Pr_{\Omega}( X(i_1) = 1 - y_1 \vee \dots \vee X(i_j) = 1 -y_j) (1 + L e p)^j \qquad \text{Proposition~\ref{flll-dist-prop-simple}} \\
&= 1 - (1 - 2^{-j}) (1 + L e p)^j \geq 2^{-j} - ( e^{e j L p} - 1) \geq 2^{-j} - e^2 j L p
\end{align*}
This only leads to $\epsilon$-approximate, $j$-wise independence for $\epsilon = \Theta(j L 2^{-k})$. Thus, it appears that the MT-distribution is asymptotically closer in variation distance to $\Omega^*$, compared to the LLL-distribution.
\subsection{Minimum implicate size for boolean formulas} 
Given a boolean formula $\Phi$, we say that a clause $C$ is a \emph{non-trivial implicate} if $\Phi \models C$ but $C$ is not a tautology; that is, any solution to $\Phi$ also satisfies the clause $C$. The problem of determining implicates of a formula $\Phi$ has numerous connections to knowledge representation and artificial intelligence; see for instance \cite{darwiche}. Our MT-distribution results imply a simple bound on implicate size for CNF formulas. To the best of our knowledge, this property has not been studied before.

\begin{proposition}
\label{implicate-prop1}
Suppose that $\Phi$ is a CNF formula, in which every clause contains at least $k$ distinct literals, and in which each variable occurs in at most $L \leq \frac{2^k}{e k}$ clauses. Then every non-trivial implicate of $\Phi$ has size at least $k - \lfloor \log_2( e L ) \rfloor \geq \log_2 k$.
\end{proposition}
\begin{proof}
By deleting literals from clauses, we may assume without loss of generality that every clause contains exactly $k$ literals. If $C$ is an implicate of $\Phi$ containing $j$ variables $X(i_1), \dots, X(i_j)$, then this implies that there is some value $y_1, \dots, y_j$ for these variables which is impossible. In particular, since the MT-distribution is supported on satisfying assignments to $\Phi$, we would have $\pmt( X(i_1) = y_1 \wedge \dots \wedge X(i_j) = y_j) = 0$. By Theorem~\ref{dist-thm}, this is only possible if $e L 2^{-k} \geq 2^{-j}$, i.e. $j \geq k - \log_2 (e L).$
\end{proof}

The bound of Proposition~\ref{implicate-prop1} can easily be seen to be tight (up to a small additive term):
\begin{proposition}
\label{implicate-prop2}
For any integers $k, L \geq 1$, there is a CNF formula $\Phi$ in which each clause contains $k$ variables, each variable appears in at most $L$ clauses, and which has a non-trivial implicate of size $j = k - \lfloor \log_2 L \rfloor$.
\end{proposition}
\begin{proof}
Let $C$ be an arbitrary clause on $j$ variables and define the formula $\Phi$ by
$$
\Phi = \bigwedge_{\substack{y_1, \dots, y_{k-j} \in \{0, 1 \}^{k-j}}} (C \vee X(1) = y_1 \vee \dots \vee X(k-j) = y_{k-j})
$$
Clearly $\Phi \models C$ and  each clause has size $k$. Furthermore, the formula $\Phi$ contains $2^{k-j}$ clauses altogether. Thus, if we set $j = k - \lfloor \log_2 L \rfloor$, then $\Phi$ will obey the restriction on the maximum number of occurrences per variable.
\end{proof}

Note that although Proposition~\ref{implicate-prop1} was proved via the MT-distribution, the result itself is phrased purely in terms of boolean logic, without any reference to the LLL or probability theory.

\section{Independent transversals avoiding subsets of vertices}
\label{find-trans-sec}
Consider a graph  $G = (V,E)$ along with a partition of its vertices into $k$ blocks as $V = V_1 \sqcup \dots \sqcup V_k$. An \emph{independent transversal (IT)} $T$ of $G$ is an independent set of $G$, with the additional property that $|T \cap V_i| = 1$ for all $i$. This is also known as an \emph{independent system of representatives}. This combinatorial structure has received significant attention, starting in \cite{bollobas-erdos-szemeredi:isr}. Of particular important is obtaining sufficient conditions and algorithms for the existence of such independent transversals, as a function of the maximum degree $\Delta$ or other graph parameters.   

Haxell  \cite{haxell:struct-indep-set} showed that if each $V_i$ has size $b \geq 2 \Delta$, then an IT exists; this condition is existentially optimal as a function of $\Delta$ \cite{Szabo2006}.   This result however uses descent arguments which could take exponential time. More recently, an algorithm of Graf, Harris, \& Haxell \cite{graf} provides an efficient randomized algorithm to find an IT when the blocks have size $b \geq (2 + \epsilon) \Delta$ for arbitrary constant $\epsilon > 0$. Furthermore, a derandomization method of Harris \cite{harris-derand} can make this deterministic.

These algorithms are quite involved, and cannot be parallelized. By contrast, the LLL provides relatively straightforward constructions of ITs, subject to a slightly stronger constraint $b \geq 4 \Delta$. This is a motivating application of the cluster-expansion LLL criterion \cite{bissacot}. Furthermore, the MT algorithm provides a corresponding simple linear-time randomized algorithm, along with a parallel algorithm when $b \geq (4 + \epsilon) \Delta$ for constant $\epsilon > 0$. The variables in this case correspond to each block, and the value of variable $i$ is the vertex $v \in V_i \cap T$. For each edge, there is a bad-event that both end-points are selected.

To simplify the formulas throughout this section, let us define $\alpha = \sqrt{1 - 4 \Delta/b}$.

\begin{proposition}[\cite{bissacot}]
\label{ftrans-prop1}
Suppose $b \geq 4 \Delta$. Then the cluster-expansion LLL criterion is satisfied with $\tilde \mu(B) = \frac{4}{(1+\alpha)^2 b^2}$ for all $B \in \mathcal B$.
\end{proposition}

In \cite{rabern}, Rabern considered a further question for such IT's: given a set of vertices $L \subseteq V$, when can we guarantee the existence of an IT satisfying $T \cap L = \emptyset$?  Rabern showed a general result for certain types of graphs where the block sizes could be different. We will consider here the case when all the blocks have a common size $b$. When $b \geq 2 \Delta$ and $|L| < b$, then the result of Rabern shows that such an IT exists. The condition on the size of $L$ is clearly optimal (otherwise one could choose $L$ to be equal to one of the blocks). The algorithms of \cite{graf} and \cite{harris-derand} could likewise provide corresponding polynomial-time algorithms to find such an IT, but again these algorithms are extremely complex, slow, and cannot be parallelized.

We analyze next how the MT algorithm is able to find such an IT avoiding $L$. Specifically, given some arbitrary set $L \subseteq V$, we show a lower bound on the probability $\pmt( L \cap T \neq \emptyset )$. We begin with a useful estimate when $L$ is a subset of a single block, based on arguments in \cite{partial-resampling}.
\begin{proposition}[\cite{partial-resampling}]
\label{fugly-prop}
Let $L$ be a subset of a block $V_r$. Define the event $E \equiv T \cap L \neq \emptyset$. Then
$$
\Psi_{\mathcal B[E]}(E) \leq \frac{2 b}{b + |L| + (b - |L|) \alpha}
$$
\end{proposition}
\begin{proof}
In $\mathcal B[E]$, we have bad-events $B_f$ for every edge $f$ of $G$, except the edges which have an endpoint in $L$. We will apply the cluster-expansion criterion to $\mathcal B[E]$; we define $\tilde \mu(B_f) = \frac{4}{b^2 (1 + \alpha)^2}$ if edge $f$ does not involve block $V_r$, and $\tilde \mu(B_f) = \beta$ if edge $f$ involves a vertex of $V_r - L$, for some parameter $\beta \leq \frac{4}{b^2 (1 + \alpha)^2}$ to be specified. As there are at most $(b-|L|) \Delta$ edges which can participate in a bad-event of $\mathcal B[E]$, the cluster-expansion criterion is satisfied in this case if
$$
\beta  \geq \frac{1}{b^2} \times (1 + b \Delta \times \frac{4}{b^2(1+\alpha)^2})  (1 + (b - |L|) \Delta \beta) \\
$$
which is satisfied by
$$
\beta  = \frac{4}{(1+\alpha) b (b + |L| + (b - |L|) \alpha)}
$$

Since $E$ is a singleton event, the formula (\ref{sb1}) gives 
$$
\Psi_{\mathcal B[E]}(E) \leq 1 + \sum_{\substack{\text{edges $f$ involving} \\ \text{some vertex in $V_r - L$}}} \mu_{\mathcal B[E]}(B_f) \leq 1+ (b-|L|) \Delta  \beta.
$$

Simple algebraic manipulations show that this is the same as the stated formula.
\end{proof}

Our main result is now to bound $\pmt(L \cap T \neq \emptyset)$:
\begin{theorem}
  \label{ftrans-prop2}
 Suppose $b \geq 4 \Delta$. Then, for any $L \subseteq V$,
$$
\pmt (L \cap T \neq \emptyset) \leq \max \Bigl( \frac{2 |L|}{b + |L| + (b-|L|) \alpha}, \frac{2 (1 - e^{-|L|/b})} {1 + \alpha} \Bigr)
$$
\end{theorem}
\begin{proof}
  
Let us define $v = |L|/b$ and we assume that $v < 1$ as otherwise this holds vacuously.  Let us write $y_i = |L \cap V_i|/b$, and let define the function
  $$
f(y) = \frac{2 y}{1 + y + (1 - y) \alpha}
$$

  For each $i = 1, \dots, k$,  define  $A_i$ to be the event that $T \cap (L \cap V_i)\neq \emptyset$. Then note that $L \cap T \neq \emptyset$ is equal to the disjunctive event $A_1 \cup \dots \cup A_m$, and so by Proposition~\ref{fx4},
\begin{align*}
\pmt(L \cap T \neq \emptyset)  \leq \sum_{i=1}^k \Pr_{\Omega}(A_i \wedge \neg A_1 \dots \wedge \neg A_{i-1}) \Psi_{\mathcal B[A_i]} (A_i)  = \sum_{i=1}^k \Psi_{\mathcal B[A_i]} (A_i) \times  y_i \prod_{j=1}^{i-1} (1 - y_j) 
\end{align*}

Proposition~\ref{fugly-prop} shows that
$$
\Psi_{\mathcal B[A_i]}(A_i) \leq \frac{2 b}{b + |L \cap V_i| + (b - |L \cap V_i|) \alpha} = \frac{2 b}{b + y_i b + (b - y_i b) \alpha} = \frac{f(y_i)}{y_i}
$$

So if we define the function
$$
S(x_1, \dots, x_k) = \sum_{i=1}^k f(x_i) \prod_{j=1}^{i-1} (1 - x_j)
$$
then we have shown that
$$
\pmt(L \cap T \neq \emptyset) \leq S(y_1, \dots, y_k).
$$

Our task is now bound the function $S$.

Define parameter $z = \max(0, \frac{1 - 3 \alpha}{1 - \alpha} )$. By rearranging and sorting the blocks, we can ensure that $$
y_{\ell} \leq y_{\ell-1} \leq \dots \leq y_2 \leq y_1 \leq z \leq y_{\ell+1} \leq y_{\ell+2} \leq \dots \leq y_k,
$$
for some integer $\ell \leq k$. 

So, for any integers $\ell, n$, let us define $Q_{n,\ell}$ to be the set of vectors $(x_1, \dots, x_n) \in [0,1]^n$ which satisfy the following constraints:
\begin{enumerate}
\item[(A1)] $z \geq x_1 \geq \dots \geq x_{\ell} \geq 0$
\item[(A2)] $z \leq x_{\ell+1} \leq x_{\ell+2} \leq \dots \leq x_n$
\item[(A3)] $x_1 + x_2 + \dots + x_n = v$.
\end{enumerate}

We also define $Q_{n} = \bigcup_{\ell=0}^n Q_{n, \ell}$, and define $V_n$ to be the maximum value of $S(x)$ over all $x$ in the compact space $Q_n$. We know that the count vector $y$ is in $Q_k$. In Appendix~\ref{math-app}, we prove that
$$
V_k \leq \max\Bigl( \frac{2 v}{1 + v + (1-v) \alpha}, \frac{2 (1 - e^{-v})}{1 + \alpha} \Bigr)
$$
which establishes the claim. The proof of this fact requires some fine-grained technical analysis of the functions $f$ and $S$, but the intuition is that the maximizing vector $x$ has the form either $(0, \dots, 0, v)$ or $(v/k, \dots, v/k)$.
\end{proof}

\begin{corollary}
\label{fcor1}
Suppose $b \geq \frac{e^2}{e-1} \Delta \geq 4.30027 \Delta$. Let $L \subseteq V$ be an arbitrary vertex set. Then the MT algorithm outputs an IT which is disjoint to $L$ with probability $\Omega(1 - |L|/b)$.
\end{corollary}
\begin{proof}
Let $v = |L|/b$ and $r = e^{-v}$. We assume $|L| < b$ as otherwise this holds vacuously. Note that the condition on $b, \Delta$ implies $\alpha \geq 1 - 2/e$.

  By Proposition~\ref{ftrans-prop2}, there is a probability of at least $1 - \max( \frac{2 v}{(1 + v + (1 - v) \alpha}, \frac{2 (1 - e^{-v})}{1 + \alpha} )$ that the MT-distribution produces such an IT. We need to show that both of these quantities are $\Omega(1 - v)$.   For this first quantity, we have
  $$
  1 - \frac{2 v}{1 + v + (1 - v) \alpha} = (1 - v) \frac{ 1 + \alpha }{1 + \alpha + v(1-\alpha)} \geq (1-v) \frac{1+\alpha}{2}
  $$

  For the second quantity, we have
  $$
  1 - \frac{2 (1 - e^{-v})}{1 + \alpha} = (1 - v) \frac{ 2 e^{-v} + \alpha - 1 }{ (1 + \alpha)(1-v)} = (1-v) \frac{2 r + \alpha - 1}{(1+\alpha) (1 + \ln r)}
    $$

  Since function $\ln r$ is concave-down and $r \geq 1/e$, we have $\ln r \leq e r - 2$ and hence $\frac{2 r + \alpha - 1}{(1+\alpha) (1 + \ln r)} \geq \frac{2 r + \alpha - 1}{(1+\alpha) (e r - 1)}$. This is a rational function of $\alpha$ and $r$, and since $r \in [1/e,1], \alpha \geq 1 - 2/e$ it is at least $\frac{1}{e - 1}$; in particular it is $\Omega(1)$.
  \end{proof}

\section{The Permutation LLL}
\label{fperm-lll-sec}
Recall that in the permutation LLL setting, we select a permutation $\pi$ uniformly from $\Omega = S_n$. There is a collection $\mathcal B$ of monomial bad-events, which all have the form $B \equiv \pi(x_1) = y_1 \wedge \dots \wedge \pi(x_r) = y_r$, for chosen values $(x_1, y_1), \dots, (x_r, y_r)$. For such an event, we say that $B$ \emph{demands} $(x,y)$ if $\pi(x) = y$ is a necessary condition for $B$ to hold, i.e. $(x_i, y_i) = (x,y)$ for some index $i$. For any pairs $(x,y), (x', y')$ we say that $(x,y) \sim (x', y')$ if $x = x'$ or $y = y'$.

For any such monomial events $A, A'$, we say $A \sim A'$ iff there are pairs $z = (x,y), z' = (x', y')$ such that $A$ demands $z$ and $A'$ demands $z'$ where $z \sim z'$. For a pair $z$, we also say that $A \sim z$ if $A$ demands $z'$ for some $z' \sim z$.

We will consider here a slightly denser dependency graph than the one used by Erd\H{o}s \& Spencer \cite{erdos-spencer}: we have an edge on bad-events $B, B'$ whenever $B \sim B'$. In this case, there is an efficient algorithm, known as the \emph{Swapping Algorithm}, which plays a similar role the MT algorithm for the variable-assignment setting. 
\begin{algorithm}[H]
\centering
\caption{The Swapping Algorithm}
\begin{algorithmic}[1]
\State Generate the permutation $\pi$ uniformly at random and independently. 
\While{there is some true bad-event}
\State {Choose an arbitrary true bad-event $B \equiv \pi(x_1) = y_1 \wedge \dots \pi(x_r) = y_r$.}
\For{$i = 1, \dots, r$}
\State Swap entry $x_i$ of $\pi$ with $x'_i$ chosen uniformly at random from $[n] - \{x_1, \dots, x_{i-1} \}$.
\EndFor
\EndWhile
\end{algorithmic}
\end{algorithm}

We refer to a single iteration of the loop in lines 4 -- 5 as \emph{resampling} $B$ and we refer to line 5 as a \emph{swap}. We let $\pi_t$ denote the value of the permutation after $t$ resampling steps. We refer to $\pi_0$, the permutation selected at line 1, as the \emph{initial configuration.}

This section will be devoted to showing tighter bounds on the MT-distribution (the distribution of states at the termination of the Swapping Algorithm). Our analysis, which is based on \cite{prev-lll}, is very technically involved and will be deferred to Appendix~\ref{permutation-appendix}. We contrast our approach here with work of \cite{achlioptas, harvey}, which develops a general and clean way of calculating probabilities for the permutation LLL. At the same time, their strategy is very abstract, covering a number of probability spaces including spanning trees and perfect matchings on a clique, and one cannot easily take advantage of the extra structure available just for permutations.

\subsection{A new MT-distribution bound}
Consider a monomial event $A \equiv \pi(x_1) = y_1 \wedge \dots \wedge \pi(x_r) = y_r$. Our goal is to show an upper bound on $\pmt(A)$. As shown in \cite{prev-lll}, there is an MT-distribution which is essentially identical to the ``generic'' LLL-distribution bound:
\begin{proposition}[\cite{prev-lll}]
\label{fprev-lll-prop}
For any monomial event $A$, we have $\pmt(A) \leq \Pr_{\Omega} \Psi_{\mathcal B[A]}(A)$.
\end{proposition}

The quantity $\Psi_{\mathcal B[A]}(A)$ here is a sum over stable subsets of $\mathcal B$. Our main result is that the sum can be restricted to sets which also satisfy an additional property known as \emph{orderability}.

\begin{definition}
We say that a set $Y \subseteq \mathcal B$ is \emph{orderable} to $A$ if there exists an ordering of $Y$ as $Y = \{B_1, \dots, B_{\ell} \}$ and a list of pairs $\{ z_1, \dots, z_{\ell} \}$,  such that $A$ demands $z_1, \dots, z_{\ell}$ such that $z_i \sim B_i, z_i \not \sim B_1, \dots, B_{i-1}$ for $i = 1, \dots, {\ell}$.
\end{definition}

We define $\text{Ord}(A)$ to be the collection of all stable, orderable sets to $A$, and we correspondingly define parameter $\Psi'$ by:
$$
\Psi'(A) = \sum_{I \in \text{Ord}(A)} \mu_{\mathcal B[A]}(I) 
$$

Our main result will be the following:
\begin{theorem}
\label{fnew-lll-prop}
For a monomial event $A$, we have $\pmt(A) \leq \Pr_{\Omega}(A) \Psi'(A)$.
\end{theorem}

Our proof strategy, as in \cite{prev-lll}, is based on generating a succinct witness tree that ``explains'' the history behind why $A$ came to be true. Suppose that we run the Swapping Algorithm to run time $T$, resampling events $B_1, \dots, B_T$. If $A$ is true at time $T$, we construct a witness tree $\hat \tau^{T,A}$ using the procedure given as Algorithm~\ref{gen-tau-alg}:
\begin{algorithm}[H]
\centering
\caption{Procedure for generating $\hat \tau^{T,A}$}
\label{gen-tau-alg}
\begin{algorithmic}[1]
\State  Let $\hat \tau^{T,A }_{ T+1}$ be a singleton node $r$ labeled $A$.
\For{$t = T, \dots, 1$}
\If{there is a  node $v \in \hat \tau^{T,A}_{ t+1} - r$ labeled by some $B \sim B_t$}
\State Select the node of $v$ of greatest depth (breaking ties by label) whose label is $B \sim B_t$.
\State Let  $\hat \tau^{T,A}_{ t}$ be $\hat \tau^{T, A}_{ t+1}$, plus one additional node, which is a child of $v$ labeled $B_t$.
\ElsIf{the children of $r$ in  $\hat \tau_{ t+1}^{T,A}$ have labels $B'_1, \dots, B'_s$ and $\{ B'_1, \dots, B'_s, B_t \} \in \text{Ord}(A)$}
\State Let  $\hat \tau^{T,A}_{ t}$ be $\hat \tau^{T,A}_{ t+1}$, plus one additional node, which is a child of $r$ labeled $B_t$.
\EndIf
\State{\textbf{else} set $\hat \tau^{T,A}_{ t} = \hat \tau^{T,A}_{ t+1}$}
\EndFor
\State Set $\hat \tau^{T,A} = \hat \tau^{T,A}_{ 1}$
\end{algorithmic}
\end{algorithm}

We say that a rooted tree whose node has label $A$ and whose non-root nodes are labeled by events in $\mathcal B$, is a \emph{tree-structure rooted in $A$}. We say that $\tau$ \emph{appears} if $A$ is true at some time $T$ and $\hat \tau^{T,A} = \tau$.  Our result will follow from the key technical lemma:
\begin{lemma}[Witness Tree Lemma]
  \label{fwitness-tree-lemma}
  For any given tree-structure $\tau$, the probability that $\tau$ appears is at most $w(\tau)$.
\end{lemma}

Let us first discuss a simple example which gives the intuition behind Lemma~\ref{fwitness-tree-lemma}. Suppose that event $A \equiv \pi(1) = 1$ becomes true during the Swapping Algorithm, and two other bad-events were first resampled, $B_1 \equiv \pi(1) = 2, B_2 \equiv \pi(2) = 1$. We want to build a witness tree for  $A$.

In the conventional method of building witness trees, the witness tree $\tau = \hat \tau^{2,A}$ would have a root node labeled $A$ with two children labeled $B_1, B_2$. Then a necessary condition for $\tau$ to appear is that the initial configuration must satisfy $\pi_0(1) = 2, \pi_0(2) = 1$. We must resample $B_1, B_2$ (in some unspecified order). After the second such resampling, we must perform the swap such that $\pi(1) = 1$. Regardless of the state of the system at the time of this second resampling, this has probability $1/n$. With a little more careful analysis, we see that the tree $\tau$ appears with probability at most $1/n^3 = w(\tau)$.

However, observe that the \emph{first} resampling among $B_1, B_2$ played essentially no role in this argument. The key events that ``cause'' $A$  only happen during the \emph{second} resampling. Thus, we should be able explain $A$ without mentioning $B_1$, giving a more ``compressed'' history for why $A$ become true. And indeed, $\hat \tau^{T,A}$ has only a \emph{single} child --- either $B_1$ or $B_2$, whichever occurred last. (Observe that $\{ B_1, B_2 \}$ is not orderable to $A$.)

Many other complications can arise for larger trees. This analysis mostly follows along the lines of \cite{prev-lll} with a few key definitions and proofs modified; we thus defer it to Appendix~\ref{permutation-appendix}. This result is quite similar to a result shown in \cite{harris-llll}; we compare the two settings in Appendix~\ref{compare-sec}. 

\subsection{Complex events and the original configuration}
When computing the probability of some complex event $A$ in the MT distribution, it is often useful to separately bound the probability that $A$ occurs in the initial configuration (which is just the uniform distribution on $S_n$), and that $A$ occurs later during the Swapping Algorithm. We say that $A$ \emph{occurs non-initially}, if when we run the Swapping Algorithm algorithm, $A$ is false on the original configuration but true in the final configuration. Given any events $A, C$, we say that $A$ \emph{occurs non-initially and before $C$ (n.i.b. $C$)}, if at some time $t > 0$ during the execution of the Swapping Algorithm such that $A$ is false on $\pi_0$, $A$ is true on $\pi_t$, and $C$ is false on $\pi_0, \dots, \pi_{t-1}$.

This is based on the following for witness trees:
\begin{proposition}
  \label{fchange-prop2}
If the tree-structure containing a single node labeled $A$ appears, then $A$ holds in the initial configuration.
\end{proposition}
\begin{proof}
  Suppose not, and so $A$ demands $(x, y)$ but $\pi_0(x) \neq y$. In order for the tree-structure to appear, we must have $\pi_{T}(x) = y$ at some later time $T > 0$. As we show later in  Proposition~\ref{fchange-prop}, some bad-event $B$ must be resampled prior to time $T$ with $B \sim (x,y)$. But then when forming the witness tree $\hat \tau^{T,A}$, a node labeled $B$ would be placed below the root node labeled $A$; in particular, the tree $\hat \tau^{T,A}$ could not be the singleton root node.
  \end{proof}

\begin{proposition}
\label{fprop2}
Let $A$ be a monomial event and $C$ be an arbitrary event. The probability that $A$ occurs n.i.b. $C$ is at most $\Pr_{\Omega}(A) (\Psi_{\mathcal B[ C ]}'(A) - 1).$
\end{proposition}
\begin{proof}
By Corollary~\ref{fchange-prop2},  a necessary condition for the singleton tree to appear is that $A$ occurs in the initial configuration. Thus, if $A$ occurs for the first time at time $t > 0$, then some \emph{non-singleton} tree-structure $\tau$ rooted in $A$ must appear.  Furthermore, since $A$ is false at all prior times, every node in the tree has a label in $\mathcal B[C]$. 

By Lemma~\ref{fwitness-tree-lemma}, the probability that $A$ occurs n.i.b. $C$ can be bounded summing over all non-singleton tree-structures rooted in $A$, whose nodes are labeled by $\mathcal B[ C ]$. By Proposition~\ref{ac1}, the total weight of all tree-structures rooted in $A$ is at most $\Pr_{\Omega}(A) \Psi'(\mathcal B[C])(A)$, and the singleton tree-structure contributes $\Pr_{\Omega}(A)$. 
\end{proof}

As an illustration, consider a disjunctive event:
\begin{theorem}
\label{swapping-dist-thm}
For a collection $\mathcal A$ of monomial events, we have
$$
\pmt(\ora) \leq \Pr_{\Omega}(\ora) + \sum_{A \in \mathcal A} \Pr_{\Omega}(A) (\Psi'_{\mathcal B[ \ora]}(A) - 1)
$$
\end{theorem}
\begin{proof}
  The probability that $\ora$ occurs in the initial configuration is $\Pr_{\Omega}(\ora)$. If $\ora$ occurs non-initially, then let $t > 0$ be the minimal time at which $\ora$ has occurred. So some $A \in \mathcal A$ is true at time $t$ and $\ora$ is false at times $0, \dots, t-1$, and so $A$ occurs n.i.b. before $\ora$. So
  \[
\Pr( \text{$\ora$ occurs non-initially} ) \leq \sum_{A \in \mathcal A} \Pr( \text{$A$ occurs n.i.b. $\ora$})  \leq \sum_A \Pr_{\Omega}(A) (\Psi_{\mathcal B[\ora]}'(A) - 1) \qedhere
\]
\end{proof}

\section{Latin transversals}
\label{latin-sec}
Consider an $n \times n$ matrix $A$, in which all the entries are assigned a color. A \emph{Latin transversal} for $A$ is a selection of cells one from each row and one from each column, so that no color is selected more than once. A canonical application of the permutation LLL, which was in fact the original motivation for the LLLL \cite{erdos-spencer}, is constructing Latin transversals when each color appears at most $\frac{27 n}{256}$ times in $A$. In this section, we extend this result with two applications: weighted Latin transversals and partial Latin transversals.

Throughout, we define $A_k \subseteq [n] \times [n]$ to be the set of cells $(x,y)$ with $A(x,y) = k$, and we assume that $|A_k| \leq \Delta$ for some parameter $\Delta$.

\subsection{Weighted Latin transversals}
\label{fweighted-Latin-sec}
Consider some weighting function $w: [n] \times [n] \rightarrow \mathbb R_{\geq 0}$. Let us define $W = \sum_{x,y} w(x,y)$ and for any set $Z \subseteq [n] \times [n]$, we define $w(Z) = \sum_{(x,y) \in Z} w(x,y)$. It is clear that there is a permutation (a selection of one cell from each row and column) of weight $W/n$. The following result shows that we can obtain a Latin transversal whose weight is not much larger than this.

\begin{proposition}
\label{vc1}
Suppose $\Delta \leq \frac{27}{256} n$. Then the Swapping Algorithm finds a Latin transversal $T$ with probability one, and $\bE_{\text{MT}}[ w(T) ] \leq \tfrac{5}{3} W/n$.
\end{proposition}
\begin{proof}
For each pair of entries $(i,j), (i', j')$ with $A(i,j) = A(i', j')$, we have a bad-event that $\pi(i) = j \wedge \pi(i') = j'$.  The cluster-expansion criterion is satisfied by setting $\tilde \mu(B) = \alpha = \frac{256}{81 n^2}$ for all $B \in \mathcal B$.

Now, consider any event $E \equiv \pi(x) = y$. An orderable set of neighbors for $E$ is either (i) the empty set, or (ii) a singleton set containing a bad-event of the form $\pi(x) = y' \wedge \pi(x'') = \pi(y'')$, or (iii) a singleton set containing a bad-event of the form $\pi(x') = y \wedge \pi(x'') = \pi(y'')$. These three cases contribute respectively $1$, $n \Delta \alpha$, and $n \Delta \alpha$ to the summation $\Psi'(E) = \sum_{I \in \text{Ord}(E)} \mu(I)$. So, by Theorem~\ref{fnew-lll-prop}, we have $\pmt(E) \leq \Pr_{\Omega}(E) (1 + 2 n \Delta \alpha) = 1/n \times (1 + 2 n \times \frac{27}{256} n \times \frac{256}{81 n^2}) = \frac{5}{3 n}$.

Summing over all pairs $(x,y)$ gives
\[
\emt[w(T)] = \sum_{x,y} w(x,y) \pmt(\pi(x) = y) \leq \frac{5}{3 n} \sum_{x,y} w(x,y) = \frac{5}{3} W/n \qedhere
\]
\end{proof}

Note that that the LLL-distribution or the MT-distribution bounds of \cite{prev-lll} show only the weaker bound $\bE[w(T)] \leq \frac{16}{9} W/n$.

\subsection{Partial Latin transversals}
\label{fpartial-Latin-sec}
When $\Delta = \beta n$ for $\beta > 27/256$, then the LLL construction of Proposition~\ref{vc1} does not work, and we do not know how to show the existence of a full Latin transversal. In \cite{mt-dist}, Harris \& Srinivasan considered algorithms to obtain a \emph{partial} Latin transversal, that is, is a selection of cells in the matrix, at most one cell per row and one cell per column, with the property that no color appears more than once. Two algorithms were analyzed. These both start by selecting a permutation $\pi \in S_n$, and end by modifying this permutation into a partial Latin transversal by deleting repeated colors.  The simplest algorithm just draws the permutation $\pi$ uniformly at random, with no resampling.  The second algorithm runs the Swapping Algorithm, but only resamples a randomly chosen subset of the cells. These algorithms achieve partial Latin transversals of expected size respectively $n \times \frac{1 - e^{-\beta}}{\beta},  n \times ( \frac{1}{2} + \sqrt[3]{\frac{27}{2048 \beta }} )$. 

We will discuss a more advanced scheme to construct partial Latin transversals. We first recall a result of \cite{stein}, which gives a lower bound on the probability that a random permutation meets a given set of entries in an array.

\begin{proposition}[\cite{stein}]
\label{fstein-prop2}
Suppose that $Z \subseteq [n] \times [n]$, and permutation $\pi$ is chosen uniformly at random. Then
$$
\Pr ( \bigwedge_{(x,y) \in Z} \pi(x) \neq y) \leq (1 - 1/n)^{|Z|} \leq e^{-|Z|/n}
$$
\end{proposition}

We will construct the partial Latin transversal in three stages. We begin by drawing a random subset $M \subseteq [m] \times [m]$, wherein each $(x,y)$ goes into $M$ independently with probability $r$ for some constant $r \in [0,1]$ to be specified. We also define $\overline M = [n] \times [n] - M$. We then run the Swapping Algorithm for a carefully chosen bad-event set $\mathcal B$, which depends on set $M$. We, when the Swapping Algorithm terminates, we are left with a permutation $\pi_{\text{final}}$. Let  $C_k$ denote the number of cells of color $k$ with $\pi_{\text{final}}(x) = y$ and let $L_k = \max(C_k - 1,0)$. By discarding repeated colors, we finally a Latin transversal of size $n - \sum_k L_k$.

First, note that simple concentration bounds give the following:
\begin{proposition}
  \label{conc1}
  Let $r, \beta \in [0,1]$ be arbitrary constants, and suppose that $|A_k| \leq \Delta$ for $\Delta = \beta n$.  Then, with probability $1 - o(1)$,  all of the following bounds hold for all indices $i,j,k$:
  \begin{enumerate}
  \item There are at most $r(1+o(1))$ values $y$ with $(i,y) \in M$.
  \item There are at most $r(1+o(1))$ values $x$ with $(x,j) \in M$.
  \item We have $|\overline M \cap A_k| \leq (1-r) \Delta (1+o(1))$ and $|M \cap A_k| \leq r \Delta (1+o(1))$
  \item We have $|\overline M| \leq (1-r) n^2 (1+o(1))$.
  \end{enumerate}
\end{proposition}

Let us assume for the remainder of this proof that $M$ satisfies all these conditions.  We will separately analyze the initial configuration $\pi_0$, and the final configuration $\pi_{\text{final}}$. The following formula is key to estimating the expected value of $L_k$:
\begin{proposition}  
  \label{term-bound}
  Suppose that $|\overline M \cap A_k| = u_k$, and define $Q_k$ to be the total number of pairs $(x_1, y_1) \in A_k, (x_2, y_2) \in A_k$ satisfying the following conditions:
  
  \begin{enumerate}
  \item[(B1)] $(x_1, y_1) < (x_2, y_2)$ (in the lexicographic order)
  \item[(B2)] It does NOT hold that $\pi_0(x_1) = y_1 \wedge \pi_0(x_2) = y_2 \wedge (x_1, y_1) \in \overline M \wedge (x_2, y_2) \in \overline M$
  \item[(B3)] $\pi_{\text{final}}(x_1) = y_1$ and $\pi_{\text{final}}(x_2) = y_2$
  \end{enumerate}

  Then we have $\bE[L_k] \leq  u_k/n  - 1 + e^{-u_k/n} + \bE[Q_k]$.
\end{proposition}
\begin{proof}
 Let us suppose there are $C$ cells $(x,y) \in A_k$ with $\pi_{\text{final}}(x) = y$, and, of these, $R$ cells $(x,y)$ satisfy the additional property that $(x,y) \in \overline M$ and $\pi_0(x) = y$. By definition $L_k = \max(C-1, 0)$. One can easily verify the inequality for integers $C \geq R \geq 0$:
  \begin{equation}
    \label{hgh1}
  \max(C - 1,0) \leq \max(R - 1,0) + \tbinom{C - R}{2} + R (C - R)
  \end{equation}
  
There are precisely $\binom{C}{2}$ pairs satisfying (B1), (B3), and precisely $\binom{R}{2}$ pairs which also satisfy the four properties $\pi_0(x_1) = y_1, \pi_0(x_2) = y_2, (x_1, y_1) \in \overline M, (x_2, y_2) \in \overline M$. Hence $Q_k = \binom{C}{2} - \binom{R}{2} = \tbinom{C - R}{2} + R (C - R)$.

 Next, define $T$ to be the total number of cells satisfying $(x,y) \in \overline M \cap A_k$ and $\pi_0(x) = y$. We have $R \leq T$, and so $\bE[ \max (R - 1, 0) ] \leq \bE[ \max(T - 1, 0) ] = \bE[T] - 1 + \Pr(T = 0)$. Since $\pi_0$ is a uniformly random permutation and $|\overline M \cap A_k| = u_k$, we have $\bE[T] = u_k/n$ and Proposition~\ref{fstein-prop2} gives $\Pr(T = 0) \leq e^{-u_k/n}$. So $\bE[ \max(R-1,0)] \leq u_k/n  - 1 + e^{-u_k/n}$.

  Thus, taking the expectation of (\ref{hgh1}), we have
  \[
  \bE[ L_k ] \leq \bE[ \max(R-1, 0)] + \bE[ \tbinom{C - R}{2} + R (C - R) ] \leq u_k/n  - 1 + e^{-u_k/n} + \bE[Q_k] \qedhere
  \]
\end{proof} 

\begin{theorem}
\label{fpartial-Latin-thm}
Let $\beta \in [27/256, 1]$. Suppose that each color appears at most $\Delta \leq \beta n$ times in the matrix. Let $q_{\text{max}} = 1-\sqrt{1 - (27/256)/\beta}$ and let $q \in [0, q_{\text{max}}]$.

There is an polynomial-time algorithm which produces an partial Latin transversal whose expected number of cells is at least $f(\beta, q) n - o(n)$, where $f(\beta, q)$ is given by the following formula:
$$
f(\beta, q) = q - \frac{ e^{-(1-q) \beta} - 1 }{\beta} - 2 \beta^2 \gamma^{5/4} (2 q- q^2) (1 - q)^2
$$
and where $\gamma$ is the smallest positive root of $\gamma - (1 + \beta (2 q - q^2) \gamma)^4  = 0$.
\end{theorem}
\begin{proof}
  Since the function $f$ is continuous and the parameters $\beta, q$ come from compact spaces, it suffices to show this for constant $\beta > 27/256$.

  We will run the Swapping Algorithm, where the bad-event set $\mathcal B$ is defined by having a separate bad-event $\pi(x_1) = y_1 \wedge \pi(x_2) = y_2$ for each unordered pair $(x_1,y_1), (x_2,y_2)$ satisfying $A(x_1,y_1) = A(x_2,y_2)$ and either $(x_1, y_1) \in M$ or $(x_2, y_2) \in M$.

  Each such event has probability $\frac{1}{n (n-1)}$. We will apply the cluster-expansion with $\tilde \mu(B) = \alpha = \gamma/n^2$ for all $B$. To calculate the cluster-expansion criterion here, let us count the number of bad-events $B'$ which demand $(x_1, y_1')$. If $(x_1, y_1') \in M$, then we may choose $(x_2', y_2')$ to be any other pair with $A(x_2', y_2') = A(x_1, y_1')$; if $(x_1, y_1') \in \overline M$, then we may choose any pair $(x_2', y_2') \in M$ with $A(x_2', y_2') = A(x_1, y_1')$. In light of Proposition~\ref{conc1}, there are $(1+t) ( r n \Delta + (1-r) n (\Delta r) ) = \beta n^2 (2 r - r^2) (1+t)$ choices $B'$, for some function $t = o(1)$.
  
  All such bad-events are neighbors in the dependency graph. There are similar choices for $B'$ to demand $(x_1', y_1), (x_2, y_2'), (x_2', y_2)$. Overall, the cluster-expansion criterion is satisfied if
  \begin{equation}
\label{frr0}
\alpha \geq \frac{1}{n(n-1)} (1 + \beta n^2 (2 r - r^2) (1+t) \alpha)^4 = (1+t') \frac{(1 + \beta n^2 (2 r - r^2)) \alpha)^4}{n^2}
  \end{equation}
  where $t' = o(1)$ as well.

  Since $\alpha = \gamma/n^2$, this holds if $\gamma \geq (1+t') (1 + \beta (2 r - r^2)) \gamma)^4$. Now observe that if $r$ is a constant smaller than $q$, then $(1 + \beta (2 r - r^2)) \gamma)^4$ is a constant strictly smaller than $(1 + \beta (2 q - q^2)) \gamma)^4$, and so the bound (\ref{frr0}) bound holds for $n$ sufficiently large. In particular, the Swapping Algorithm terminates with the choice of parameter $r = q - o(1)$.

  We next bound the expected value of $Q_k$ in Proposition~\ref{term-bound} for some color $k$. Let us consider some pair $z_1 = (x_1, y_1) \in A_k, z_2 = (x_2, y_2) \in A_k$ satisfying (B1), and let us define the event $E_{z_1,z_2}$ that $\pi(x_1) = y_1 \wedge \pi(x_2) = y_2$. Note that if $(x_1, y_1) \in M$ or $(x_2, y_2) \in M$, then $E$ would be chosen as a bad-event in $\mathcal B$, and so $E_{z_1,z_2}$ could not hold in the final configuration.

  Thus, a necessary condition for $(x_1, y_1), (x_2, y_2)$ to satisfy (B3) is that $(x_1, y_1) \in \overline M$ and $(x_2, y_2) \in \overline M$. But in this case, in order to satisfy (B2), we must have $\pi_0(x_1) \neq y_1$ or $\pi_0(x_2) \neq y_2$, i.e. $E_{z_1,z_2}$ does not hold in the initial configuration $\pi_0$. Overall, we see that
  $$
  \bE[Q_k] \leq \sum_{\substack{ (x_1, y_1) <  (x_2, y_2) \\ (x_1, y_1) \in \overline M \cap A_k \\ (x_2, y_2) \in \overline M \cap A_k}} \Pr( \text{$E_{z_1,z_2}$ holds non-initially})
  $$
  
If we let $u_k = |\overline M \cap A_k|$, then there are precisely $\binom{u_k}{2}$ summands here. Now consider some fixed pair $z_1 = (x_1, y_1), z_2 = (x_2, y_2)$ and we compute the probability that $E = E_{z_1,z_2}$ holds non-initially using Proposition~\ref{fprop2}. We have $\Pr_{\Omega}(E) =\frac{1}{n(n-1)}$. An orderable set of neighbors to $E$ may contain up to one bad-event $B$ which overlaps on $(x_1, y_1)$  and up to one bad-event $B$ which overlaps on $(x_2, y_2)$. By Proposition~\ref{conc1}, there are at most $n \Delta (2 r - r^2) (1+o(1))$ bad-events which demand $(x_1, y_1')$, and similarly for other three cases, and so
\begin{align*}
  \Psi'(E) &\leq (1 + 2 n \Delta (2 r - r^2) \alpha + o(1))^2 \leq (1 + 2 \beta (2 q - q^2) \gamma)^2 (1+o(1)).
\end{align*}

Thus Proposition~\ref{fprop2} gives
$$
\Pr(\text{$E$ occurs non-initially}) \leq (1+o(1)) \frac{ (1 + 2 \beta (2 q - q^2) \gamma)^2 - 1}{n^2}
$$

Since this probability holds for any of the $\binom{u_k}{2}$ choices for $z_1, z_2$, Proposition~\ref{term-bound} gives
$$
  \bE[L_k] \leq (1+o(1)) \Bigl( u_k - 1 + e^{-u_k/n} + \binom{u_k}{2}  \frac{ (1 + 2 \beta (2 q - q^2) \gamma)^2 - 1}{n^2} \Bigr)
$$

This is an increasing concave-up function of $u_k$, and we know that $u_k \leq (1+o(1)) (1-r) \Delta \leq (1+o(1)) (1-q) \beta n$. Thus, we have
\begin{align*}
  \bE[L_k] &\leq (1+o(1)) \frac{u_k}{(1-q) \beta n} \Bigl( (1-q) \beta n - 1 + e^{-(1-q) \beta} + \frac{ (1-q)^2 \beta^2 ((1 + 2 \beta (2 q - q^2) \gamma)^2 - 1)}{2} \Bigr)
\end{align*}

Summing over $k$ and noting $\sum_k u_k = | \overline M | \leq (1+o(1)) (1 - q) n^2$, we have
$$
  \bE[ \sum_k L_k ] \leq (1+o(1)) n \Bigl( (1-q) + \frac{e^{-(1-q) \beta} - 1}{\beta} + \frac{ (1-q)^2 \beta^2 ( (1 + 2 \beta (2 q - q^2) \gamma)^2 - 1) }{2 \beta} \Bigr)
  $$
  
Thus, the expected number of remaining cells in the final partial Latin transversal is given by
\[
n - \bE[\sum_k L_k] \geq n (1 - o(1)) \Bigl( q - \frac{e^{-(1-q) \beta} - 1}{\beta} - \frac{ (1-q)^2 \beta^2 ( (1 + 2 \beta (2 q - q^2) \gamma)^2 - 1)}{2 \beta} \Bigr)
\]

Simple algebraic manipulations, and using the identity $(1 + \beta (2 q - q^2) \gamma) = \gamma^{1/4}$, gives the stated formula.
\end{proof}

For any given $\beta > 0$, one may select $q \in [0, q_{\text{max}}]$ to maximize the resulting $f(\beta, q)$. We let $g(\beta)$ denote this quantity $g(\beta) = \max_{q \in [0, q_{\text{max}}]} f(\beta, q)$. Our algorithm thus can obtain a partial Latin transversal with $g(\beta) n - o(n)$ cells. For any fixed value of $\beta$ one can numerically optimize $f(\beta, q)$ and compute (a lower bound on) $g(\beta)$. We compare this result with the two algorithms of \cite{mt-dist} for selected values of $\beta$; there is a tiny but definite improvement, by up to $0.4 \%$.

\begin{center}
\begin{tabular}{|c||c|c|c|} \hline  
$\beta$ & Theorem~\ref{fpartial-Latin-thm} & $\frac{1 - e^{-\beta}}{\beta}$ & $\frac{1}{2} + \sqrt[3]{\frac{27}{2048 \beta}}$ \\
\hline
\hline
0.11 & 0.9938 & 0.9470 & 0.9930 \\
0.12 & 0.9810 & 0.9423 & 0.9789 \\
0.13 & 0.9692 & 0.9377 & 0.9663 \\
0.14 & 0.9584 & 0.9331 & 0.9550 \\
0.15 & 0.9483 & 0.9286 & 0.9446  \\
0.16 & 0.9390 & 0.9241 & 0.9351  \\
0.17 & 0.9304 & 0.9196 & 0.9264  \\
0.18 & 0.9226 & 0.9152 & 0.9184  \\
0.19 & 0.9154 & 0.9107 & 0.9109 \\
0.20 & 0.9088 & 0.9063 & 0.9040  \\
0.21 & 0.9023 & 0.9020 & 0.8974 \\
0.22 & 0.8977 & 0.8976 & 0.8913  \\
0.23 & 0.8933 & 0.8933 & 0.8856 \\
0.24 & 0.8890 & 0.8890 & 0.8801  \\
0.25 & 0.8848 & 0.8848 & 0.8750  \\
\hline
\hline
\end{tabular}
\end{center}

For $\beta \geq 1/4$, the function $f(\beta, q)$ is maximized at $q = 0$, so $g(\beta) = f(\beta, 0) = \frac{1 - e^{-\beta}}{\beta}$. In these cases we are not using the Swapping Algorithm at all. However for $\beta < 1/4$ we can see that $f(\beta, q)$ is strictly larger than either of the other two estimates (for $\beta \geq 0.22$ the difference is below the third decimal point).

\section{Acknowledgments}
Thanks to Aravind Srinivasan for helpful discussions and feedback. Thanks to Landon Rabern for pointing me to some of the literature on independent transversals. Thanks to anonymous journal reviewers for helpful suggestions and comments.

\appendix

\section{Proof of Theorem~\ref{lll-dist-theta-thm}}
\label{mt-dist-result}
\begin{lemma}
\label{mu-lem}
For any $S \subseteq \mathcal B$, we have $\Pr_{\Omega}(\overline S \mid \overline {\mathcal B - S}) \geq \frac{1}{\sum_{J \subseteq S}  \mu_{\mathcal B}(J)}$
\end{lemma}
\begin{proof}
For any $V \subseteq \mathcal B$, let us define
$$
\tilde Q(V) = \sum_{\substack{I \subseteq V \\ \text{$I$ stable}}} \prod_{B \in I} (-\Pr_{\Omega}(B))
$$

Observe that $\Pr_{\Omega}(\overline S \mid \overline {\mathcal B - S}) = \frac{ \Pr_{\Omega}( \overline{ \mathcal B} ) }{ \Pr_{\Omega} ( \overline{ \mathcal B - S } ) }$. As shown in \cite{shearer}, for any sets $S_1 \subseteq S_2 \subseteq \mathcal B$ we have 
$\frac{ \Pr_{\Omega} (\overline S_1) }{\tilde Q(S_1)} \leq \frac{ \Pr_{\Omega}( \overline S_2) }{\tilde Q(S_2)}$. Thus, setting $S_1 = \mathcal B - S$ and $S_2 = \mathcal B$, we have 
$\Pr_{\Omega}(\overline S \mid \overline {\mathcal B - S}) \geq \frac{ \tilde Q(\mathcal B)}{ \tilde Q( \mathcal B - S)}$. 

Observe that $\tilde Q(\mathcal B) = \sum_{\text{$I$ indepedent}} (-\Pr_{\Omega}(B)) = Q_{\mathcal B}(\emptyset)$. Also, we can expand:
{\allowdisplaybreaks
\begin{align*}
\sum_{I \subseteq S} Q_{\mathcal B}(I) &= \sum_{I \subseteq S} \sum_{\substack{J: I \subseteq J \subseteq \mathcal B \\ \text{$J$ stable}} }(-1)^{|J| - |I|} \prod_{B \in J} \Pr_{\Omega}(B) = \sum_{\substack{J \subseteq \mathcal B \\ \text{$J$ stable}}} \prod_{B \in J} (-\Pr_{\Omega}(B)) \sum_{I \subseteq J \cap S} (-1)^{-|I|} \\
&= \sum_{\substack{J \subseteq \mathcal B \\ \text{$J$ stable} \\ J \cap S = \emptyset}} \prod_{B \in J} (-\Pr_{\Omega}(B)) = \sum_{\substack{J \subseteq \mathcal B - S \\ \text{$J$ stable}}} \prod_{B \in J} (-\Pr_{\Omega}(B)) = \tilde Q( \mathcal B - S)
\end{align*}
}

Thus, $\sum_{I \subseteq S} \mu_{\mathcal B}(I) = \frac{\sum_{I \subseteq S} Q_{\mathcal B}(I)}{Q_{\mathcal B}(\emptyset)} = \frac{\tilde Q( \mathcal B - S)}{\tilde Q(\mathcal B)}$. 
\end{proof}
{
\renewcommand{\thetheorem}{\ref{lll-dist-theta-thm}}
\begin{theorem}
For any event $E$, we have $\Pr_{\Omega} (E \mid \overline{\mathcal B}) \leq \Pr_{\Omega}(E) \Psi_{\mathcal B[E]}(E)$.
\end{theorem}
\addtocounter{theorem}{-1}
}
\begin{proof}
Let $S = \mathcal B[E]$ and $U = \mathcal B - S$. We use Bayes' Theorem, taking into account that events in $U$ are all subsets of $E$:
\begin{align*}
\Pr_{\Omega} (\overline E \mid \overline{\mathcal B}) &= \frac{\Pr_{\Omega}( \overline U \mid \overline E, \overline S) \Pr_{\Omega}(\overline E \mid \overline S)}{\Pr_{\Omega}(\overline U \mid \overline S)} = \frac{\Pr_{\Omega}(\overline E \mid \overline S)}{\Pr_{\Omega}(\overline U \mid \overline S)} \geq \Pr_{\Omega}(\overline E \mid \overline S)
 \end{align*}
 
To finish, we need to show that $\Pr_{\Omega}(E \mid \overline S) \leq \Pr_{\Omega}(E) \Psi_{\mathcal B[E]}(E)$. By Lemma~\ref{mu-lem}, we have:
\begin{align*}
\Pr_{\Omega}(E \mid \overline S) &= \frac{\Pr_{\Omega}( E, \overline{N(E) \cap S} \mid \overline{S - N(E)})}{\Pr_{\Omega}(\overline{N(E) \cap S} \mid \overline{S - N(E)})} \leq  \frac{\Pr_{\Omega}( E)}{\Pr_{\Omega}(\overline{N(E) \cap S} \mid \overline{S - N(E)})} \\
&\leq  \Pr_{\Omega}( E) \sum_{J \subseteq N(E) \cap S} \mu_{\mathcal B[E]}(J) \leq  \Pr_{\Omega}( E) \Psi_{\mathcal B[E]}(E) \qedhere
\end{align*}
\end{proof}

\section{Analysis for Theorem~\ref{ftrans-prop2}}
\label{math-app}
Recall that we have defined parameters $v = |L|/b < 1, \alpha = \sqrt{1 - 4 b/\Delta} \in [0,1], z = \max(0, \frac{1 - 3 \alpha}{1 - \alpha} )$ as well as the function $f(y) = \frac{2 y}{1 + y + (1 - y) \alpha}, S(x_1, \dots, x_k) = \sum_{i=1}^k f(x_i) \prod_{j=1}^{i-1} (1 - x_j)$.

Let us first observe that $V_n \leq V_{n+1}$ for any integer $n$. For, given any vector $(x_1, \dots, x_n) \in Q_{n, \ell}$, the vector $x' = (x_1, \dots, x_{\ell}, 0, x_{\ell+1}, \dots, x_n)$ is in $Q_{n+1, \ell+1}$, and has $S(x') = S(x)$.

Let us now consider some integer $n$ and choose some vector $x \in Q_n$ to maximize $S(x)$. This exists by compactness of $Q_n$. There may be multiple such maximizing vectors, so we use the tie-breaking rules of selecting $x$ to \emph{minimize} the number of coordinates $i$ with $x_i = z$.

We first claim that $x \in Q_{n, n-1} \cup Q_{n,n}$. For, let $\ell$ be maximal such that $x \in Q_{n,\ell}$; if $\ell \geq n-1$ then we are done. If $\ell < n-1$, then $z \leq x_{n-1} \leq x_n$. If $z = 0$, then $x_{n-1} > 0$, as otherwise we would have $x_{\ell} = \dots = x_{n-1} = 0$, in which case $x \in Q_{n,n-1}$. Finally, by property (A3), we have $x_{n-1} + x_n \leq v \leq 1$.

Now define the vector $x' = (x_1, \dots, x_{\ell}, 0, x_{\ell+1}, \dots, x_{n-1} + x_n)$, and observe that $x' \in Q_{n, \ell+1}$. Furthermore, we have
$$
S(x') - S(x) = \Bigl( \prod_{i=1}^{n-2} (1 - x_i/b) \Bigr) \Bigl( f(x_{n-1} + x_n) - ( f(x_{n-1}) + f(x_n) (1 - x_{n-1}/b) \Bigr)
$$

We claim that for any real numbers $w_1, w_2$ satisfying $w_1 \geq z, w_2 \geq w_1 > 0, w_1 + w_2 \leq 1$ we have $f(w_1)  + f(w_2) (1 - w_1) < f(w_1 + w_2)$. This algebraic inequality can be verified by using standard algorithms for decidability of real-closed fields.\footnote{To verify this using the Mathematica computer algebra system, use:
  \vspace{-0.1in}
\begin{scriptsize}
\begin{verbatim}
f[y_] = 2 y/(1 + (1 - y) alpha + y)
z = (1 - 3 alpha)/(1 - alpha)
Reduce[f[w1] + f[w2] (1 - w1) >= f[w1 + w2] && w1 >= z && w2 >= w1 && 0 <= alpha < 1/3]
Reduce[f[w1] + f[w2] (1 - w1) >= f[w1 + w2] && w1 > 0 &&  w2 >= w1 && (w1 + w2) <= 1 && alpha >= 1/3]
\end{verbatim}
\end{scriptsize}
} Since $x_{n-1}, x_n$ satisfy these conditions (playing the role of $w_1, w_2$ respectively), we have $S(x') > S(x)$, which is a contradiction to our choice of $x$.

We next claim that $x_1 = x_2 = \dots = x_{n-1}$. For, suppose that there is some index $p < n-1$ with $x_p > x_{p+1}$. Since $x \in Q_{n,n-1} \cup Q_{n,n}$, necessarily $z \geq x_p > x_{p+1}$. Note that this must imply that $\alpha < 1/3$, as otherwise $z = 0$. Now consider the vector $x'$ obtained by replacing $x_p$ and $x_{p+1}$ by the mean value $u = \frac{x_p + x_{p+1}}{2}$. Since $x \in Q_{n,\ell}$ then also $x' \in Q_{n, \ell}$, and furthermore we have:
\begin{align*}
S(x') - S(x) &= \Bigl( \prod_{j=1}^{p-1} 1 - x_j \Bigr) \Bigl(  f(u) + f(u)(1 - u) -  f(x_p) - f(x_{p+1}) (1 - x_p)
\\
& \qquad \qquad
+\sum_{i=p+2}^n f(x_i) \bigl( (1 - u)^2 -  (1-x_p)(1-x_{p+1}) \bigr) \Bigr)
\end{align*}

Clearly $(1-u)^2 \geq (1 - x_p)(1-x_{p+1})$. Also, by (A3), we must have $x_j < 1$ for $j \leq p-1$.  Furthermore, we can mechanically verify the following algebraic inequalities: for $z \geq w_1 \geq w_2 \geq 0$ it holds that
$$
f(w_1) + f(w_2)(1 - w_1) \leq f( \frac{w_1 + w_2}{2}) + f( \frac{w_1 + w_2}{2} ) (1 - \frac{w_1 + w_2}{2})
$$
and furthermore, if $w_1 < z$, this inequality is strict.\footnote{To verify in Mathematica:
\vspace{-0.1in}  
\begin{scriptsize}
\begin{verbatim}
f[y_] = 2 y/(1 + (1 - y) alpha + y)
z = (1 - 3 alpha)/(1 - alpha)
Reduce[f[w1] + f[w2] (1 - w1) >  f[(w1 + w2)/2] (2 - (w1 + w2)/2) && z >= w1 > w2 >= 0 && 0 <= alpha < 1/3]
Reduce[f[w1] + f[w2] (1 - w1) >= f[(w1 + w2)/2] (2 - (w1 + w2)/2) && z > w1 > w2 >= 0 && 0 <= alpha < 1/3]
\end{verbatim}
\end{scriptsize}}

As a consequence of these facts, we have $S(x') \geq S(x)$, and the inequality is strict unless $x_p = z$. If $x_p < z$, then this contradicts that $x$ maximizes $S(x)$; if $x_p = z$, this contradicts that $x$ minimizes the number of coordinates with $x_i = z$.

Thus, we have shown that $x$ must have the form $x = (t, \dots, t, u)$. By Property (A3) we have $u \in [0,v]$ and $t = \frac{v - u}{n-1}$. So
$$
S(x) = \sum_{i=1}^{n-1} f(t) (1 - t)^{i-1} + f(u) (1-t)^{n-1} = f \Bigl( \frac{v-u}{n-1} \Bigr) \frac{n-1}{v - u} (1 - (1 - \tfrac{v-u}{n-1})^n) + f(u) (1 - \tfrac{ v - u}{n-1})^{n-1}
$$

Thus, we have shown that
$$
V_n \leq \max_{u \in [0,v]}  f \Bigl( \frac{v-u}{n-1} \Bigr) \frac{n-1}{v  - u } (1 - (1 - \tfrac{v-u}{n-1})^{n-1}) + f(u) (1 - \tfrac{ v - u}{n-1})^{n-1}
$$
Since $V_n \leq V_{n+i}$ for all $i \geq 0$, we have:
$$
V_k \leq \inf_{n \geq k} V_{n} \leq \inf_{n+1 \geq k} \max_{u \in [0,v]} f \Bigl(\frac{v-u}{n} \Bigr) \frac{n}{v  - u } (1 - (1 - \tfrac{v-u}{n})^{n}) + f(u) (1 - \tfrac{ v - u}{n})^{n}
$$

For any given value $u \in [0,v]$, the quantity $f \bigl(\frac{v-u}{n} \bigr) \frac{n}{v  - u } (1 - (1 - \frac{v-u}{n})^{n}) + f(u) (1 - \frac{ v - u}{n})^{n}$ converges pointwise to $\frac{2 (1-e^{u-v})}{1 + \alpha} + f(u) e^{u-v}$ as $n \rightarrow \infty$. Since $u$ is maximized over a compact domain $[0,v]$, the maximization commutes with the pointwise limit, and so we have
\begin{align*}
  V_k &\leq \max_{u \in [0,v]} \lim_{n \rightarrow \infty} f \Bigl(\frac{v-u}{n} \Bigr) \frac{n}{v  - u } (1 - (1 - \tfrac{v-u}{n})^{n}) + f(u) (1 - \tfrac{ v - u}{n})^{n} \\
  &= \max_{u \in [0,v]} \frac{2 (1-e^{u-v})}{1 + \alpha} + \frac{2 u e^{u-v}}{1 + u + \alpha(1-u)}
\end{align*}

If we set $r = e^u, s = e^v$, then with simple algebraic substitutions we have shown that
$$
V_k \leq \max_{r \in [1,s]} g(r) \qquad \text{where we define } g(r) := \frac{2 (1 - r/s)}{1 + \alpha} + \frac{2 \ln r \times r/s}{1 + \ln r + (1-\ln r) \alpha}
$$

To finish the proof, we claim that $g(r)$ attains its maximum value at one of the endpoints $r = 1$ or $r = s$. Since $g(1) = \frac{2 (1 - e^{-v})}{1 + \alpha}$, while $g(s) = \frac{2 v}{1 + v + (1 - v) \alpha}$, this will show the desired bound on $V_n$.

Now note that $g'(1) = 0$ and $g''(1) = \frac{6 \alpha - 2}{(1 + \alpha)^2 s}$ and $g''(r)$ has a single root at $r = r_0 = e^{\frac{3 \alpha - 1}{\alpha -1}}$.

If $\alpha < 1/3$, we see that $g''(1) < 0$, and so the function $g$ is initially decreasing, and possibly later increasing for $r \geq r_0$. This implies that its maximum value must occur at one of the endpoints $r = 1$ or $r = s$.

Likewise, if $\alpha > 1/3$, then $g''(1) > 0$ and that $r_0 < 1$, which implies that the function $g$ is increasing throughout its domain. so its maximum value occurs at the endpoint $r = s$.

Finally, if $\alpha = 1/3$, then $g''(1) = 0, g'''(1) = \frac{3}{4 s} > 0$. Again, this implies that the function $g$ is increasing throughout its domain and its maximum value occurs at the endpoint $r = s$.

\section{Comparison with the work of \cite{harris-llll}}
\label{compare-sec}
The definition of an orderable set for permutations is inspired by a similar criterion of of \cite{harris-llll} for the variable-assignment LLL. There is one major difference between Theorem~\ref{fnew-lll-prop} and the result of \cite{harris-llll}. Here, as we build witness trees, we only enforce the condition that the children of the \emph{root node} are orderable; in \cite{harris-llll}, it was required that \emph{for every node labeled $B$ in the witness tree}, the children of that node were orderable to $B$. 

While our analysis for permutations only cuts down the space of witness trees by a constant factor, the latter work cuts down it down by a factor which is exponential in the size of the tree.  The overall convergence of the MT algorithm is determined by the \emph{overall growth rate} of witness trees as a function of their size. Thus, \cite{harris-llll} yields a stronger overall bound on the convergence of the MT algorithm, which is not possible for us.

We make the following conjecture that a similar witness tree lemma should still be possible in the permutation LLL setting, where the orderability conditioned is enforced at all the nodes:
\begin{conjecture}
\label{ford-conj}
Suppose that one builds witness trees while enforcing the condition that, for every node $v \in \tau$ labeled $B$, the children of $v$ receive distinct labels $B_1, \dots, B_s$, such that $\{ B_1, \dots, B_s \}$ is orderable to $B$. Then for any tree-structure $\tau$ we have $\Pr(\text{$\tau$ appears}) \leq w(\tau)$.
\end{conjecture}

We remark that Swapping Algorithm and/or the witness tree generation process must be changed in some way in order for Conjecture~\ref{ford-conj} to hold. If the same procedure as in Algorithm~\ref{gen-tau-alg} is used directly, then \cite{kolm-personal-com} has shown counter-examples to Conjecture~\ref{ford-conj}.

Conjecture~\ref{ford-conj} would immediately give the following result:
\begin{conjecture}
\label{ford-conj2}
Suppose that the weighting function $\tilde \mu: \mathcal B \rightarrow [0, \infty)$ satisfies the criterion
$$
\forall B \in \mathcal B \qquad \tilde \mu(B) \geq \Pr_{\Omega}(B) \sum_{\substack{T \subseteq N(B) \\ \text{$T$ orderable to $B$}}} \prod_{A \in T} \tilde \mu(A)
$$

Then the Swapping Algorithm terminates with probability one; the expected number of resamplings of $B$ is at most $\tilde \mu(B)$.
\end{conjecture}

Conjecture~\ref{ford-conj2} would yield stronger bounds for Latin transversals, hypergraph packings, and other applications. As an example:
\begin{corollary}
Suppose that Conjecture~\ref{ford-conj2} holds. Then, for any $n \times n$ colored array $A$, in which each color appears at most $\Delta = n/8$ times, there is a Latin transversal of $A$ with no repeated colors.
\end{corollary}
\begin{proof}
For each quadruple $(i_1, j_1, i_2, j_2)$ with $A(i_1, j_1) = A(i_2, j_2)$ we have a separate bad-event. Define $\tilde \mu(B) = \alpha$ for all bad-events. Then Conjecture~\ref{ford-conj2} reduces to showing that 
$$
\alpha \geq \frac{1}{n(n-1)} (1 + 2 n (\Delta-1) \alpha)^2.
$$ 
This can be satisfied for $\Delta = n/8$.
\end{proof}

\section{Proof of Theorem~\ref{fnew-lll-prop}}
\label{permutation-appendix}

We will break down our overall analysis into three stages.
\begin{enumerate}
\item We transform a given tree-structure $\tau$ into a \emph{witness subdag}; this is a similar object to the witness tree, but instead of giving the history of resamplings of bad-events, it gives a history of individual swappings. We define and describe some structural properties of these graphs.
\item We define the \emph{future-subgraph} at time $t$, denoted $G_t$. This is a kind of graph which encodes conditions on $\pi_t$ which are necessary in order for $\tau$ to appear.  We analyze how a future-subgraph $G_t$ imposes conditions on the corresponding permutation $\pi_t$.
\item  We compute the probability that the swapping satisfies these conditions over time.
\end{enumerate}

\subsection{Witness subdags}
\begin{proposition}
\label{ind-prop1}
Suppose that $\tau$ is a witness tree, and that distinct nodes $v, v' \in \tau$ have the same depth. Then $L(v) \not \sim L(v')$.
\end{proposition}
\begin{proof}
We may assume that $v, v'$ are not the root node. Thus, $L(v), L(v')$ are both bad-events (as opposed to the justifying event $A$). The earlier event would be eligible to placed as a child of the later one, thus will be placed either there or lower in the tree.
\end{proof}

\begin{definition}[Witness subdag]
A \emph{witness subdag} is defined to be a directed acyclic simple graph, whose nodes are labeled with tuples $(B,x,y)$; if a node $v$ is labeled by $(B,x,y)$, we write $v \approx (x,y)$. This graph must in addition satisfy the following properties:
\begin{enumerate}
\item If for distinct nodes $v, v'$ we have $v \approx (x,y) \sim (x', y') \approx v'$, then either there is a path from $v$ to $v'$ or a path from $v'$ to $v$.
\item Every node of $G$ has in-degree at most two and out-degree at most two.
\end{enumerate}
\end{definition}

The witness subdags are derived from witness trees in the following manner.
\begin{definition}[Projection of a witness tree]
For a witness tree $\tau$, we define the \emph{projection of $\tau$} (denoted $\text{Proj}(\tau)$), as follows.

Suppose we have a node $v \in \tau$ labeled by an event $E = \{ (x_1, y_1), \dots, (x_r, y_r) \}$. For each $i = 1, \dots, r$, we create a corresponding node $v'_i$ labeled $(E,x_i, y_i)$ in the graph $\text{Proj}(\tau)$. 

The edges of $\text{Proj}(\tau)$ are formed follows. For each node $v' \in \text{Proj}(\tau)$, labeled by $(E, x,y)$ and corresponding to $v \in \tau$, we find the node $w_x \in \tau$ (if any) which satisfies the following properties:
\begin{enumerate}
\item[(P1)] The depth of $w_x$ is smaller than the depth of $v$.
\item[(P2)] $w_x$ is labeled by some $E'$ which demands $(x, y')$.
\item[(P3)] Among all vertices satisfying (P1), (P2), the depth of $w_x$ is maximal.
\end{enumerate}

If this node $w_x \in \tau$ exists, then it corresponds to a node $w_x' \in \text{Proj}(\tau)$ labeled $(x, y')$; we construct an edge from $v'$ to $w_x'$. By Proposition~\ref{ind-prop1}, the labels in any level of the witness tree form a stable set, thus there can be at most one such $w_x$ and at most one such $w_x'$.

We similarly define a node $w_y$, where instead of requiring $E'$ to demand $(x,y')$ we require $E'$ to demand $(x', y)$. If this node exists, we create an edge from $v'$ to the corresponding $w_y' \in \text{Proj}(\tau)$ labeled $(x',y)$. Note that it is possible to have  $w_x = w_y$.
\end{definition}

Since edges in $\text{Proj}(\tau)$ correspond to \emph{strictly} smaller depth in $\tau$, the graph $\text{Proj} (\tau)$ is acyclic.

\begin{definition}[Alternating paths] Given a witness subdag $G$, we define an \emph{alternating path} in $G$ to be a simple path which alternately proceeds forward and backward along the directed edges of $G$. For a vertex $v \in G$, the \emph{forward (respectively backward) path} of $v$ in $G$, is the maximal alternating path which includes $v$ and all the forward (respectively backward) edges emanating from $v$. (These paths are unique since $G$ has in-degree and out-degree at most two). Note that if $v$ is a source node, then its backward path contains just $v$ itself.
\end{definition}

One type of alternating path, which is referred to as the \emph{W-configuration}, plays a particularly important role.

\begin{definition}[The W-configuration]
Suppose $v \approx (x, y)$ has in-degree at most one, and the backward path contain an \emph{even} number of edges, terminating at vertex $v' \approx (x', y')$. We refer to this alternating path as a \emph{W-configuration}. (See Figure~\ref{ffig1}.)

Any W-configuration can be written (in one of its two orientations) as a path of vertices labeled $$
(x_0, y_1), (x_1, y_1), (x_2, y_1), \dots, (x_s, y_s), (x_s, y_{s+1});
$$ here the vertices $(x_1, y_1), \dots, (x_s, y_s)$ are at the ``base'' of the W-configuration. Note here that we have written the path so that the $x$-coordinate changes, then the $y$-coordinate, then $x$, and so on. When written this way, we refer to $(x_0, y_{s+1})$ as the \emph{endpoints} of the W-configuration.

Note that if $v \approx (x,y)$ is a source node, then it defines a W-configuration with endpoints $(x,y)$.
\end{definition}

\begin{figure}[H]
\begin{center}
\includegraphics[trim = 0.5cm 21.5cm 6.5cm 5cm,scale=0.5,angle = 0]{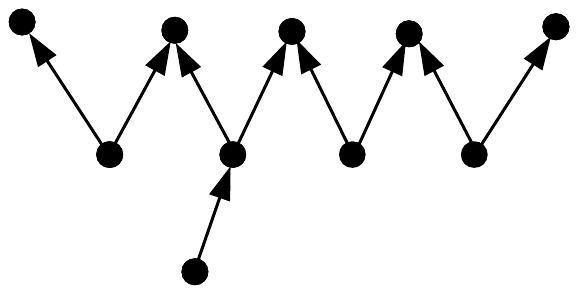}
\put(-200,70){$(x_0,y_1)$}
\put(-180,15){$(x_1,y_1)$}
\put(-135,-2){$(x', y')$}
\put(-70,70){$(x_4,y_5)$}
\caption{The vertices labeled $(x_0,y_1), (x_1, y_1), \dots, (x_4, y_5)$ form a W-configuration of length 9 with endpoints $(x_0, y_5)$. Note that the vertex $(x',y')$ is \emph{not} part of this W-configuration.
\label{ffig1}}
\end{center}
\end{figure}

\subsection{The future-subgraph and conditions on $\pi$ over time}
\begin{definition}[The future-subgraph]
\label{ffuture-defn}
For integers $0 \leq t \leq T$ and event $A$, we define the \emph{future-subgraph} $G_t^{T, A}$ as $G_t^{T,A} = \text{Proj}(\hat \tau^{T,A}_{ t})$.
\end{definition}

\begin{proposition}
\label{ffuture-prop1}
Let $T \geq 0$ and an event $A$ be given. Suppose at time $t$ we resample bad-event $B \equiv \pi(x_1) = y_1 \wedge \dots \wedge \pi(x_r) = y_r$. Then $G_{t+1}^{T,A}$ can deduced from $G_{t}^{T,A}$, according to the following rule:
If $G_t^{T,A}$ contains source nodes $v_1, \dots, v_s$ labeled $(B, x_1, y_1), \dots, (B, x_s, y_s)$, then $G^{T,A}_{t+1} = G_t - v_1 - \dots - v_s$; otherwise $G^{T,A}_{t+1} = G_t$.
\end{proposition}
\begin{proof}
By our rule for forming witness trees, $\hat \tau_{ t}$ is either equal to $\hat \tau_{ t+1}$ plus an additional node $u$ labeled by $B$, or is equal to $\hat \tau_{ t+1}$.

In this first case, we create new nodes $v_1, \dots, v_s$ in $\text{Proj}(\hat \tau_{ t})$ labeled by $(B, x_1, y_1), \dots, (B, x_s, y_s)$. By our rule for constructing edges in $\text{Proj}(\hat \tau_{ t})$, nodes $v_1, \dots, v_s$.  Thus, $\text{Proj}(\hat \tau_{ t}) = \text{Proj}(\hat \tau_{ t+1})$ plus the additional nodes $v_1, \dots, v_s$ and the claim holds.

Next suppose that $\hat \tau_{ t} = \hat \tau_{ t+1}$. We claim that $\text{Proj}(\hat \tau_{ t+1})$  cannot contain any source node $v$ labeled $(B, x,y)$. For, such a node $v$ would correspond some node $w$ in $\hat \tau_{ t+1}$ labeled $B$. Then $B$ would be eligible to be placed as a child of $w$ in forming $\hat \tau_t$, so that $\hat \tau_{ t}$ would contain an additional node compared to $\hat \tau_{ t+1}$.
\end{proof}

\begin{proposition}[\cite{prev-lll}]
\label{fchange-prop}
Suppose $\pi_{t_0}(x) \neq y$ and $\pi_{t_2}(x) = y$ for some $t_2 > t_0$. Then some bad-event $B$ must have been resampled at an intermediate time $t_1 \in (t_0, t_2 ]$, with $B \sim (x,y)$.
\end{proposition}

We next show how a given value for $G_t^{T,A}$ implies certain conditions on $\pi_t$ (irrespective of $T,A$). The following result is where we use our definition of orderability; it is the main way our proof diverges from \cite{prev-lll}.
\begin{proposition}
\label{ffuture-prop2a}
For any $T \geq 0$ and any event $A$ and any time $t \leq T$, the permutation $\pi_t$ satisfies the following condition: \emph{If $w$ is a source node $G_t^{T,A}$ with $w \approx (x,y)$, then $\pi_t(x) = y$.}
\end{proposition}
\begin{proof}
Let us write $G = G_t^{T,A}$ and $\hat \tau = \hat \tau^{T, A}_{ t}$. The node $w \in G$ corresponds to a node $v \in \hat \tau$ labeled by an event $E$ (the event $E$ may be a bad-event, or may be $A$ itself if $v$ is the root node.)

Suppose $\pi_t(x) \neq y$. In order for $\hat \tau$ to contain a node labeled $E$, we must at some point $t' > t$ have $\pi_{t'} (x) = y$; let $t'$ be the minimal such time. By Proposition~\ref{fchange-prop}, we must encounter a bad-event $B$ demanding $(x', y') \sim (x,y)$ at some intervening time $t'' < t'$.  If $x = x'$ and $y = y'$, then this implies $\pi_{t''}(x) = y$, contradicting minimality of $t'$. So $B$ demands $(x', y')$ where either $x' = x$ or $y' = y$ (but not both).

Let $\tau'' = \hat \tau^{T,A}_{ t'' + 1}$. We claim that $\tau''$ must contain some node deeper than $v$ labeled by $B$. This will imply that $\hat \tau$ contains a node labeled by $B$ at greater depth than $v$, where $B$ demands$(x', y') \sim  (x,y)$. This will correspond to a node $w' \in G$ labeled $(B, x', y') \in G$, which will have a path to $w$, contradicting that $w$ is a source node of $G$.

If $v$ is not the root node, this is immediate: the bad-event $B$ would be eligible to be a child of $v$ in $\tau''$, and so would be placed into $\tau''$ deeper than $v$. So we suppose that $v$ is the root node. Since the labels of the children of the root node must be orderable to $A$, we can let $B''_1, \dots, B''_r$ be the labels of the children of $v$ ordered such that such there is some $z_i$ demanded by $A$ with $z_i \in B''_i, z_i \not \sim B''_1, \dots, B''_{i-1}$ for $i = 1, \dots, r$.

If $B''_i$ demands $(x'', y'') \sim (x',y')$ for some $i \in [r]$, then $B \sim B''_i$, and so $B$ would be eligible to be placed into  $\tau''$ as a child of $B''_i$. Otherwise, suppose that $(x', y') \not \sim B''_i$ for all such $i$. In this case, $B$ is distinct from $B''_1, \dots, B''_r$ and $\{B''_1, \dots, B''_r, B \}$ is an orderable set to $A$ (with the ordering $B''_1, \dots, B''_r, B$ and $z_{r+1} = (x', y')$). Thus, $B$ would be eligible to be placed as a child of $v$.
\end{proof}

\begin{proposition}
\label{ffuture-prop2}
For any $T \geq 0$ and any event $A$ and any time $t \leq T$, the permutation $\pi_t$ satisfies the following condition: \emph{For every W-configuration in $G_t^{T,A}$ with endpoints $(x_0, y_{s+1})$, we have $\pi_t(x_0) = y_{s+1}$.}
\end{proposition}
\begin{proof}
  To simplify the notation, let us fix $T, A$ and write $G_i$ as short-hand for $G_i^{T,A}$, as well as $G = G_t = G_{t}^{T,A}$. Let us also write $\hat \tau = \hat \tau^{T, A}_{ t}$.

  We prove the claim by induction on $s$. The base case is $s = 0$; in this case $G$ contains a source node $w \approx (x_0, y_1)$, and this is precisely Proposition~\ref{ffuture-prop2a}. For the induction step $s \geq 1$, consider a W-configuration with base nodes $v_1, \dots, v_s$, and let $(x_0, y_1, x_1, y_2, \dots, x_s, y_{s+1})$ be the corresponding labels. Define $Z = \{ (x_1, y_1), \dots, (x_s, y_s) \}$.

  As $s \geq 1$, the nodes $v_1, \dots, v_s$ are not sink nodes and so correspond to non-root nodes of $\hat \tau$, and hence to resampled bad-events. Let $t' \geq t$ be minimal such at time $t'$ we resample a bad-event $B$ where one of the nodes $v_i$ has label $B$ and $v_i$ is a source node of $G_{t'}$. (This must exist, as otherwise the nodes $v_1, \dots, v_s$ would persist in the graph $G_{T+1}$). Let $\{ (x_{i_1}, y_{i_1}), \dots (x_{i_r}, y_{i_r}) \}$ be the set of pairs in $Z$ which are demanded by $B$. By definition of $B$ and $t'$ we have $r \geq 1$.

  By minimality of $t'$, the nodes $v_{i_1}, \dots, v_{i_r}$ are all source nodes in $G_{t'}$. The induction hypothesis therefore gives $\pi_{t'} (x_{i_1}) =  y_{i_1}, \dots, \pi_{t'}(x_{i_r}) = y_{i_r}$. The updated $G_{t'+1}$ (which is derived from $G_{t'}$ by removing $v_{i_1}, \dots, v_{i_r}$, plus possibly some additional source nodes of $G_{t'+1}$), has $r+1$ new W-configurations of size strictly smaller than $s$. By inductive hypothesis, the updated permutation $\pi_{t'+1}$ must then satisfy $\pi_{t'+1}(x_0) = y_{i_1}, \pi_{t'+1}(x_{i_1}) = y_{i_2}, \dots, \pi_{t'+1}(x_{i_r}) = y_{s+1}$.

As shown in \cite{prev-lll}, we may suppose without loss of generality that the resampling of $B$ swaps $x_{i_1}, \dots, x_{i_r}$ in that order, and then performs other swaps involving other elements of $B$. Let $\sigma$ denote the permutation after swapping $x_{i_1}, \dots, x_{i_r}$; we have $\pi_{t'+1}(x_{i_1}) = \sigma(x_{i_1}), \dots, \pi_{t'+1}(x_{i_r}) = \sigma(x_{i_r})$. Evidently $x_{i_1}$ has swapped with $x_{i_2}$, then $x_{i_2}$ has swapped with $x_{i_3}$, and so on, until eventually $x_{i_r}$ has swapped with $x'' = \pi_{t'}^{-1} y_{s+1}$. 

Thus, $\sigma(x'') = y_{i_1}$. If the resampling of $B$ has no further swaps involving $x''$, then $\sigma(x'') = \pi_{t'+1}(x'')$; as $\pi_{t'+1}(x_0) = y_{i_1}$, this implies that $x'' = x_0$.

If, on the other hand, the resampling of $B$ causes any further swaps involving $x''$, then this will result in $\pi_{t'+1}(x') = y_{i_1}$ where $(x',y') \in B$. This would imply that $x_0 = x'$, that is, $B$ demands $(x_0, y')$ where $y' \neq y_1$. But this would contradict that in the W-configuration in $G_{t'}$ the node labeled $(x_0, y_1)$ has in-degree one. 

Thus, we see that $x'' = x_0$, and there are no further swaps during the resampling of $B$ which affect $x''$. Consequently, $x_0 = x'' = \pi_{t'}^{-1} y_{s+1}$. This shows that $\pi_{t'}(x_0) = y_{s+1}$.

We finally claim that $\pi_t(x_0) = y_{s+1}$. For, by Proposition~\ref{fchange-prop}, otherwise we would have encountered some bad-event $B''$ which demands $(x', y') \sim (x_0, y_{s+1}$ at some time $t'' < t'$. This bad-event $B'$ would give rise to nodes $(x_0, y')$ or $(x', y_{s+1})$ in $G_{t+1}$, which contradicts that the nodes in W-configuration labeled $(x_0, y_1)$ and $(x_s, y_{s+1})$ have in-degree one.
\end{proof}

Proposition~\ref{ffuture-prop2} can be viewed equally as a definition:
\begin{definition}[Active conditions of a future-subgraph]
We refer to the conditions implied by Proposition~\ref{ffuture-prop2} as the \emph{active conditions}. More formally, we define
$$
\text{Active}(G) = \{ (x, y)  \mid \text{ $(x,y)$ are the end-points of a $W$-configuration of $G$} \}
$$

We define $a(G) = | \text{Active}(G) |$.  For any tree-structure $\tau$, we define $a(\tau) = a( \text{Proj}(\tau) )$.
\end{definition}

\subsection{The probability that the swaps are all successful}
\label{ftotal-prob-sec}
We have so far determined necessary conditions for the permutations $\pi_t$, depending on the graphs $G_t^{T,A}$. In this section, we finish by computing the probability that the swapping subroutine causes the permutations to, in fact, satisfy all such conditions.  Proposition~\ref{fexchange-prop2}, which we quote from \cite{prev-lll}, states the key randomness condition satisfied by the swapping subroutine. 
\begin{proposition}[\cite{prev-lll}]
\label{fexchange-prop2}
Let $G$ be a fixed witness subdag, and suppose that $B$ is resampled at some time $t \leq T$. Then, conditional on all past events, the probability that $G_t^{T,A} = G$ and that $\pi_{t+1}$ satisfies the active conditions of $G_{t+1}^{T,A}$ is at most $\Pr_{\Omega}(B) \frac{ (n - a(G_{t+1}^{T,A}))! }{(n - a(G))!}.$
\end{proposition}

We finally have all the pieces necessary to prove Lemma~\ref{fwitness-tree-lemma}.
{
\renewcommand{\thetheorem}{\ref{fwitness-tree-lemma}}
\begin{lemma}
    For any given tree-structure $\tau$, the probability that $\tau$ appears is at most $w(\tau)$.
\end{lemma}
\addtocounter{theorem}{-1}
}
\begin{proof}
The Swapping Algorithm, as we have defined it, begins by selecting the permutations uniformly at random. One may also consider fixing the permutations to some arbitrary (not random) value, and allowing the Swapping Algorithm to execute from that point onward. We refer to this as \emph{starting at an arbitrary state of the Swapping Algorithm.} 

We will prove the following by induction on $\tau$: starting at an arbitrary state of the Swapping Algorithm, the probability of $\tau$ appearing is given by
\begin{equation}
\label{fwt1}
\Pr(\text{$\tau$ appears}) \leq  w(\tau) \frac{  n!}{(n - a(\tau))!}. 
\end{equation}

The base case is when $\tau$ is the singleton node labeled $A$ where $A \equiv \pi(x_1) = y_1 \wedge \dots \wedge \pi(x_r) = y_r$. Then $w(\tau) = \frac{(n-r)!}{n!}$ and $\text{Proj}(\tau)$ contains $r$ isolated nodes labeled $(A,x_i,y_i)$, so that $a(\tau) = r$. Thus the RHS of (\ref{fwt1}) is $1$ and so (\ref{fwt1}) holds vacuously. 

For the induction step, a necessary condition for $\tau$ to appear is that we resample a bad-event $B$ which is the label of a node $v \in \tau$.  Suppose we condition on that $v$ is the first such node, resampled at time $t$ and that $\tau = \hat \tau^{T,A}_{ t}$. Let $G = \text{Proj}(\tau)$; thus, a necessary condition for $t$ to appear is to have $G_t^{T,A} = G$. By Proposition~\ref{ffuture-prop2}, $\pi_t$ must satisfy $\text{Active}(G)$ and $\pi_{t+1}$ must satisfy $\text{Active}(G')$ where $G' = G_{t+1}^{T,A}$; note that $G'$ is uniquely determined from $G$ (irrespective of $T$).

By Proposition~\ref{fexchange-prop2}, $\pi_{t+1}$ satisfies these conditions with probability at most  $\Pr_{\Omega}(B) \frac{ (n - a(G'))!}{(n - a(G))!}$.

Next, if this event occurs, then subsequent resamplings must cause $\hat \tau_{ t+1}^{T,A} =  \tau - v$. Thus $G' = \text{Proj}(\hat \tau_{ t+1}^{T,A}) = \text{Proj}(\tau - v)$. To bound the probability of this, we use the induction hypothesis. Note that the induction hypothesis gives a bound conditional on \emph{any} starting configuration of the Swapping Algorithm, so we may multiply these probabilities. Since $a(G') = a(\tau - v)$ and $w(\tau) = w(\tau- v) \Pr_{\Omega}(B)$, we have
\begin{align*}
\Pr(\text{$\tau'$ appears}) &\leq \Pr_{\Omega}(B) \frac{ (n - a(G'))!}{(n - a(G))!} \times  w(\tau - v) \frac{ n!}{(n - a( \tau - v))!} = w (\tau) \frac{n!}{(n - a(\tau))!}
\end{align*}
completing the induction argument.

We now consider the necessary conditions to produce the \emph{entire} tree-structure $\tau$, and not just fragments of it. First, the original configuration $\pi_0$ must satisfy the active conditions of $G_t^{T,A} = \text{Proj}(\hat \tau^{T,A}) = \text{Proj}(\tau)$. This occurs with probability $\frac{(n - a(\tau))!}{n!}$. Next, the subsequent sampling must be compatible with $\tau$; by (\ref{fwt1}) this has probability at most $w(\tau)  \times \frac{n!}{(n - a(\tau))!}$. Again, note that the bound in (\ref{fwt1}) is conditional on any starting position of the Swapping Algorithm, hence we may multiply these probabilities:
\begin{align*}
\Pr(\text{$\tau$ appears}) &\leq \frac{(n - a(\tau))!}{n!} \times w(\tau) \frac{n!}{(n - a(\tau))!} = w(\tau) \qedhere
\end{align*}
\end{proof}

{
\renewcommand{\thetheorem}{\ref{fnew-lll-prop}}
\begin{theorem}
For any monomial event $A$, we have $\pmt(A) \leq \Pr_{\Omega}(A) \Psi'(A)$.
\end{theorem}
}
\begin{proof}
If $A$ occurs for the first time at time $T$, then $\hat \tau^{T,A}$ appears. Thus, a necessary condition for $A$ to hold is that some tree-structure $\tau$ rooted in $A$ must appear. For any fixed $\tau$ this has probability at most $w(\tau)$ by Lemma~\ref{fwitness-tree-lemma}. So the overall probability of $A$ is at most the sum of the weights of all such tree-structures. 

Now suppose that $\tau$ is some tree-structure with root node $r$ labeled $A$, and that $Y$ is the set of labels of the children of $r$. We must have $Y \in \text{Ord}(A)$.  Also, the labels of the events in $\tau$ must come from $\mathcal B[A]$, as this is the first time that $A$ has occured.

Because of the tie-breaking rules for the formation of witness trees, one can check that there is a one-to-one correspondence between possible values for the tree-structure $\tau - r$ and stable-set sequences for $\mathcal B[A]$ rooted at $Y$. Therefore, summing over $Y \in \text{Ord}(A)$ and summing over $\text{Stab}_{\mathcal B[A]}(Y)$, we get that the total weight of all such tree-structures is given by
\[
\Pr_{\Omega}(A) \sum_{Y \in \text{Ord}(A)} \sum_{S \in \text{Stab}_{\mathcal B[A]}(Y)} w(S) = \Pr_{\Omega}(A) \sum_{Y \in \text{Ord}(A)} \mu_{\mathcal B[A]}(Y) = \Pr_{\Omega}(A) \Psi'(A) \qedhere
\]
\end{proof}

\end{document}